\documentclass[12pt,a4paper]{amsart}

\usepackage{amsmath,amsthm}
\usepackage{amssymb}
\usepackage{mathtools}
\mathtoolsset{showonlyrefs}

\usepackage{latexsym}
\usepackage{hyperref}
\usepackage[all,cmtip,arrow,2cell]{xy}

\usepackage{xr}

\usepackage[active]{srcltx}
\usepackage{cite}
\usepackage{enumitem}

\usepackage[margin=3cm]{geometry}
\usepackage{eucal}

\usepackage{tikz-cd}
\usepackage{tikz-3dplot}
\usepackage{pgfplots}
\usetikzlibrary{positioning}

\theoremstyle{plain}
\newtheorem{Theorem}{Theorem}[section]
\newtheorem{Proposition}[Theorem]{Proposition}

\newtheorem{Corollary}[Theorem]{Corollary}

\theoremstyle{definition}
\newtheorem{Remark}[Theorem]{Remark}
\newtheorem{Definition}[Theorem]{Definition}

\newtheorem{Question}[Theorem]{Question}
\newtheorem{Example}[Theorem]{Example}

\newtheorem{Terminology}[Theorem]{Terminology}

\newtheorem{Notation}[Theorem]{Notation}

\newtheorem*{AxiomE1}{Axiom (E1)}
\newtheorem*{AxiomE2}{Axiom (E2)}
\newtheorem*{AxiomE3}{Axiom (E3)}
\newtheorem*{AxiomE4}{Axiom (E4)}
\newtheorem*{AxiomE5}{Axiom (E5)}

\theoremstyle{remark}

\DeclareMathOperator{\Der}{Der}

\DeclareMathOperator{\Hom}{Hom}

\DeclareMathOperator{\End}{End}

\DeclareMathOperator{\Empty}{{\_\!\_\,}}
\DeclareMathOperator{\Comma}{\downarrow}

\DeclareMathOperator*{\Lim}{lim}
\DeclareMathOperator*{\Colim}{colim}

\DeclareMathOperator{\Lan}{Lan}

\DeclareMathOperator{\intHom}{\underline{\mathcal{D}\mathrm{fl}}\mathrm{g}}

\newcommand{\bbQ}{\mathbb{Q}}
\newcommand{\bbR}{\mathbb{R}}
\newcommand{\bbZ}{\mathbb{Z}}

\newcommand{\calA}{\mathcal{A}}
\newcommand{\calB}{\mathcal{B}}
\newcommand{\calC}{\mathcal{C}}

\newcommand{\calF}{\mathcal{F}}

\newcommand{\calI}{\mathcal{I}}
\newcommand{\calJ}{\mathcal{J}}

\newcommand{\calX}{\mathcal{X}}

\newcommand{\frakg}{\mathfrak{g}}
\newcommand{\op}{{\mathrm{op}}}

\newcommand{\id}{\mathrm{id}}
\newcommand{\Id}{\mathrm{1}}
\newcommand{\pr}{\mathrm{pr}}

\newcommand{\Lie}{\mathcal{L}}

\newcommand{\Set}{{\mathcal{S}\mathrm{et}}}

\newcommand{\Mfld}{{\mathcal{M}\mathrm{fld}}}
\newcommand{\Top}{{\mathcal{T}\mathrm{op}}}

\newcommand{\Eucl}{{\mathcal{E}\mathrm{ucl}}}
\newcommand{\Dflg}{{\mathcal{D}\mathrm{flg}}}
\newcommand{\Elst}{{\mathcal{E}\mathrm{lst}}}

\newcommand{\eFun}[1]{{#1}}

\newcommand{\dgAlg}{{\mathrm{dg}\mathcal{A}\mathrm{lg}}}

\newcommand{\Lanyy}{{\mathbb{L}}}

\renewcommand{\phi}{\varphi}
\renewcommand{\epsilon}{\varepsilon}

\newcommand{\Hide}[1]{}
\renewcommand{\Hide}[1]{{\color{blue}#1}}

\tikzset{bolli/.style={
    very thick,
    fill=lightgray,
    scale=1.3,
    >=Stealth}}

\tikzset{point/.style={
    shape=circle,
    draw,
    fill=black,
    inner sep=1.5pt]}}

\tikzset{whitepoint/.style={
    shape=circle,
    draw,
    fill=white,
    inner sep=1.5pt]}}

\tikzset{redpoint/.style={
    shape=circle,
    fill=red,
    inner sep=1.85pt]
    }}

\begin{document}

\title{Elastic Diffeological Spaces}

\author[C.~Blohmann]{Christian Blohmann}
\address{Max-Planck-Institut f\"ur Mathematik, Vivatsgasse 7, 53111 Bonn, Germany}
\email{blohmann@mpim-bonn.mpg.de}


\date{\today}


\begin{abstract}
We introduce a class of diffeological spaces, called elastic,
on which the left Kan extension of the tangent functor of smooth manifolds defines an abstract tangent functor in the sense of Rosick\'{y}. On elastic spaces there is a natural Cartan calculus, consisting of vector fields and differential forms, together with the Lie bracket, de Rham differential, inner derivative, and Lie derivative, satisfying the usual graded commutation relations. Elastic spaces are closed under arbitrary coproducts, finite products, and retracts. Examples include manifolds with corners and cusps, diffeological groups and diffeological vector spaces with a mild extra condition, mapping spaces between smooth manifolds, and spaces of sections of smooth fiber bundles.
\end{abstract} 
\maketitle


\section{Introduction}

\subsection{The quest for a Cartan calculus on diffeological spaces}

A category that contains smooth manifolds as a full subcategory but has better properties, such as having all limits, colimits, and exponential objects, is often called a convenient setting for differential geometry. The price to be paid for this convenience is that such categories are usually too large as to allow for strong geometric results that hold for all its objects. A typical example is the category of diffeological spaces. It is a quasi-topos with all its good categorical properties. But since it contains arbitrary quotients, arbitrary subsets, and arbitrary intersections of smooth manifolds, topological spaces, vector spaces of arbitrary cardinality, and much more, a theorem that holds for all diffeological spaces must hold in all these cases. So instead of trying to prove statements that would have to cover this impossible generality of situations, the task is often to identify conditions that are strong enough to prove a desired result, but weak enough to allow for a wide range of examples and applications. 

A considerable part of the infinitesimal differential geometric computations on a smooth manifold $M$ can be carried out in its Cartan calculus, which consists of the tangent bundle $TM \to M$, the Lie bracket of vector fields, the graded algebra of differential forms $\Omega(M)$, together with the de Rham differential $d$, the inner derivative $\iota_v$ and the Lie derivative $\Lie_v$ for every vector field $v$, which satisfy the relations
\begin{equation*}
\begin{gathered}[]
  [d, d] = 0 \,,\quad 
  [\iota_v, \iota_w] = 0 \,,\quad 
  [\iota_v, d] = \Lie_v \,,
  \\
  [\Lie_v, \iota_w] = \iota_{[v,w]} \,,\quad
  [\Lie_v,d] = 0 \,,\quad
  [\Lie_v, \Lie_w] = \Lie_{[v,w]} \,, 
\end{gathered}
\end{equation*}
where the bracket is the graded commutator of graded derivations of $\Omega(M)$. For example, local definitons and calculations of symplectic geometry can typically be worked out in the Cartan calculus, such as hamiltonian vector fields, Poisson brackets, hamiltonian actions, Dirac structures, generalized complex geometry, contact structures, the $L_\infty$-algebra of a multisymplectic structure, homotopy momentum maps, infinitesimal models for equivariant cohomology, etc. Another example is local Lagrangian Field Theory, where the derivation of the Euler-Lagrange equations, local symmetries, Noether's theorems, the theory of Jacobi fields, etc.~take place in the Cartan calculus of the infinite jet bundle, also known as the variational bicomplex \cite{DeligneFreed:1999}. In fact, the question we will address in this paper was motivated by current developments in Lagrangian Field Theory and geometric deformation theory, areas where the basic geometric objects, spaces of fields and paths in a moduli space of structures, are naturally equipped with diffeologies. 

\begin{Question}
What are the conditions a diffeological space must satisfy so that it is equipped with a natural Cartan calculus?
\end{Question}

Of course, there are always the tautological conditions which promote the desired outcome to axioms, in our case the existence of a Cartan calculus. The task is to identify a set of conditions that is minimal or at least so small that it can be verified in a wide range of cases.

The basic structures of the Cartan calculus on a smooth manifold $M$, the differential graded algebra of differential forms $\Omega(M)$ and the tangent bundle $TM \to M$, are local, that is, they can be defined first on the open subsets $U \subset \bbR^n$ of a chart $U \subset M$ and then glued together on an atlas. More precisely, the functor $U \mapsto \Omega(U)$ is a sheaf and the tangent functor $U \mapsto TU$ a cosheaf, so that
\begin{align*}
  \Omega(M) 
  &\cong \Lim_{U \to M} \Omega(U)
  \\
  TM &\cong \Colim_{U \to M} TU
  \,,
\end{align*}
where the limit in differential graded algebras and the colimit in manifolds are taken over the category of charts of a maximal atlas of $M$. For a diffeological space $X$ we simply replace the category of charts by the category of plots and obtain the definitions
\begin{align*}
  \Omega(X) 
  &:= \Lim_{U \to X} \Omega(U)
  \\
  TX &:= \Colim_{U \to X} TU
\end{align*}
of the de Rham complex and the tangent diffeological space of $X$.

The maps $X \mapsto \Omega(X)$ and $X \mapsto TX$ are the pointwise left Kan extensions from smooth manifolds to diffeological spaces, which shows that they are functorial in $X$. The left Kan extension also applies to natural transformations, such as the bundle projection $\pi_U: TU \to U$ and the zero section $0_U: U \to TU$. This suggests, that the left Kan extension is the natural method to generalize the Cartan calculus of smooth manifolds to diffeological spaces.

How do we define the Lie bracket of vector fields on a diffeological space? The first guess is to start from the Lie algebras $\calX(U) = \Gamma(U,TU)$ of vector fields on all plots $U \to X$. However, $U \mapsto \calX(U)$ is not a functor, so that the left Kan extension cannot be applied. We could map the vector fields to the space of derivations of $C^\infty(X) = \Omega^0(X)$, which is equipped with the commutator bracket. However, this map is generally not injective, and even if it is, its image may not be closed under the bracket. Worse, the map $X \mapsto \Der(C^\infty(X))$ is still not a functor, so that this does not solve the problem of naturality. The conclusion is that the spaces of vector fields on plots are not a good starting point for the construction of a natural Cartan calculus on diffeological spaces.

Fortunately, the situation has been analyzed carefully by Rosick\'y who has identified the natural structure of the tangent functor that is needed to define the Lie bracket of vector fields. He defines an abstract tangent structure on a category $\calC$ to be an endofunctor $T: \calC \to \calC$ together with the natural transformations of the bundle projection $\pi_X: TX \to X$, zero section $0_X: X \to TX$, fiberwise addition $+_X: TX \times_X TX \to TX$, exchange of order of differentiation $\tau_X: T^2 X \to T^2 X$, and inclusion of the tangent fibers into the vertical tangent space $\lambda_X: TX \to T^2 X$, which have to satisfy a rather long list of axioms (Definition~\ref{def:TangentStructure}). It is instructive to see how all these structures come together to define the Lie bracket~\eqref{eq:BracketDef} of vector fields, avoiding any reference to the commutator bracket of derivations of some structure ring.

For us, the main advantage of Rosick\'y's approach is that all the structure is given by functors and natural transformations, to which we can apply the left Kan extension. However, this does not yield an abstract tangent structure on all diffeological spaces. The main issue is that the pointwise left Kan extension, which is given by a colimit, does not preserve limits, in particular the pullback on which the fiberwise addition of tangent vectors is defined. More precisely, the natural morphism
\begin{equation}
\label{eq:theta2Intro}
  \Colim_{U \to X} TU \times_U TU
  \longrightarrow
  TX \times_X TX
  \,,
\end{equation}
is not an isomorphism for all diffeological spaces $X$. In fact, this map is generally neither surjective nor injective, as the following two examples show.

\begin{Example}[Axis cross of the plane]
Consider the subset $\{ (x,y) \in \bbR^2 ~|~ xy = 0\} \subset \bbR^2$ with the subspace diffeology. The two tangent vectors at the orgin in the direction of the $x$-axis and the $y$-axis cannot be represented on the same plot (Figure~\ref{fig:NonElastic}). It follows that~\eqref{eq:theta2Intro} is not surjective.
\end{Example}

\begin{Example}[Folded line]
Consider the diffeological quotient space of the action $\bbZ_2 \times \bbR \to \bbR$, $(k,x) \mapsto kx$, where $\bbZ_2 = \{1,-1\}$. The quotient map $\bbR \to \bbR/\bbZ_2$ is a plot. The tangent vectors $(0,1)$ and $(0,-1)$ on its domain represent the same tangent vector on $\bbR/\bbZ_2$. This implies that the pairs $\zeta = \bigl( (0,1), (0,1)\bigr)$ and $\eta = \bigl( (0,1), (0,-1)\bigr)$ in $T\bbR \times_\bbR T\bbR$ represent the same pair of tangent vectors in $T(\bbR/\bbZ_2) \times_{\bbR/\bbZ_2} T(\bbR/\bbZ_2)$. Since the tangent morphism of every morphism of plots preserves the sum of a pair of tangent vectors at a point and since the sum of $\zeta$ is zero but that of $\eta$ is not, the two pairs cannot be equivalent in $\Colim_{U \to \bbR/\bbZ_2} TU \times_U TU$. We conclude that~\eqref{eq:theta2Intro} is not injective.
\end{Example}

\subsection{The axiom of elasticity}

Only if~\eqref{eq:theta2Intro} is an isomorphism, the left Kan extension of the addition $+_U$ on on plots is a morphism $+_X: TX \times_X TX \to TX$ that can be viewed as a fiberwise addition of tangent vectors on the diffeological space $X$. Therefore, requiring~\eqref{eq:theta2Intro} to be an isomorphism is the first condition we have to impose for a diffeological space to have a natural Cartan calculus.

A $k$-form in $\Omega(X)$ is a family of $k$-forms on all plots $U \to X$ that are compatible with the pullbacks along morphisms of plots. A vector field, however, is not represented by a family of vector fields on the plots. For this reason, there is no natural operation of inner derivative on $\Omega(X)$. For the inner derivative, we have to define a $k$-form as a fiberwise multilinear and antisymmetric morphism
\begin{equation*}
  \alpha: 
  \underbrace{TX \times_X \ldots \times_X TX}_{ =: T_k X}
  \longrightarrow \bbR
  \,.
\end{equation*}
(We avoid defining a tensor product, which would entail the usual technical issues of completion when the fibers are infinite-dimensional.) The notation $T_k X$ for the $k$-fold fiber product is standard in the literature on abstract tangent structures. The inner derivative of $\alpha$ with respect to a vector field $v: X \to TX$ is then given by precomposition
\begin{equation*}
  \iota_v \alpha :
  T_{k-1} X \xrightarrow{~\cong~} X \times_X T_{k-1} X
  \xrightarrow{~v \times_X \id~}
  T_k X \xrightarrow{~\alpha~}
  \bbR
  \,.
\end{equation*}
If we define forms as maps $T_k X \to \bbR$, how can we define the differential? The differential of a function $f: TX \to X$ is given by the tangent map,
\begin{equation*}
  df:  TX \xrightarrow{~Tf~} T\bbR 
  \xrightarrow{~\cong~} \bbR \times \bbR
  \xrightarrow{~\pr_2~} \bbR
  \,.
\end{equation*}
However, the functions and exact 1-forms do not generate the ring of forms, so that this construction cannot be extended to higher forms.

We are now in the following dilemma. Either we define differential forms as families of forms on the plots, in which case we have a differential but no inner derivative. Or we define them as fiberwise multilinear and antisymmetric morphisms $T_k X \to \bbR$, in which case we have an inner derivative, but no differential. The way out is to require that the two notions of differential forms coincide. 

We have already imposed the condition that~\eqref{eq:theta2Intro} is an isomorphism, which induces an isomorphism
\begin{equation}
\label{eq:Intro02}
  \Hom(TX \times_X TX, \bbR)
  \xrightarrow{~\cong~}
  \Lim_{U \to X} \Hom( TU \times_U TU, \bbR )
  \,.
\end{equation}
It is easy to see that this isomorphism is equivariant with respect to the exchange of the two fractors for the fiber product. Moreover, the maps are fiberwise multilinear on $T_k X$ if and only if they are on all $T_k U$. This shows that the isomorphism~\eqref{eq:Intro02} induces an isomorphism from fiberwise multilinear and antisymmetric morphisms on $TX \times_X TX$ to $\Omega^2(X)$. Since we need such an isomorphism for forms of arbitrary degree $k$, we have to impose the following axiom:
\begin{AxiomE1}
 The natural morphisms
\begin{equation*}
  \theta_{k,X}: 
  \Colim_{U \to X} T_k U \longrightarrow T_k X
  \,,
\end{equation*}
are isomorphisms for all $k > 1$.
\end{AxiomE1}

This axiom has the following geometric interpretation. Every tangent vector $v_x \in T_x X$ is represented by a path. One can picture this by stretching out $x$ in the direction of $v_x$ to a smooth path $\gamma: (-\epsilon, \epsilon) \to X$ of short but non-zero length through $\gamma(0) = x$, such that the coordinate tangent vector $\tfrac{\partial}{\partial t}$ at the origin of the interval is mapped by $T_0\gamma$ to $v_x$. In this sense, every point of a diffeological space has some elasticity in a single infinitesimal direction. 

However, we generally cannot simultaneously stretch out $x$ in the directions of several tangent vectors $v^1_x, \ldots, v^k_x \in TX_x$. That is, we cannot always find a plot $p: U \to X$ with $p(0) = x$ such that $(T_0 p) \tfrac{\partial}{\partial t^i} = v_x^i$, where $(t^1, \ldots, t^k)$ are the canonical coordinates of $U \subset \bbR^k$. And even if we can find such a plot, it may happen that the tangent map $Tp$ is not injective at $0$, so that we cannot identify the tangent vectors on $X$ with the coordinate vectors on $U$. This identification is possible at every point $x \in X$ if and only if the morphism $\theta_{k,X}$ is a bijection. If in addition we want this condition to be compatible with the smooth structure, then we have to make the stronger assumption that $\theta_{k,X}$ is an isomorphism of diffeological spaces. In this sense, Axiom~(E1) captures the geometric idea of the ``elasticity'' of a diffeological space in which any finite set of tangent directions can be streched out to a smooth ``membrane'' given by the image of a plot.

\begin{Example}[Pasta diffeologies]
We can equip a smooth manifold $M$ with an alternative diffeology by defining the plots be all smooth maps $p: U \to M$ such that the rank of $Tp: TU \to TM$ is everywhere less than or equal to $r$. Since (i) the precomposition of $p$ with a smooth function $f$ does not increase the rank, (ii) the rank is a local property, and (iii) the rank of constant maps is zero, this defines a diffeology, which we call the rank-$r$-restricted diffeology.

For $r=0$ we obtain the discrete diffeology. If $r=1$, then every plot factors through $\bbR$, so that we obtain the \textbf{Spaghetti diffeology} \cite[Sec.~1.10, footnote 1]{Iglesias-Zemmour:Diffeology}. The case $r = 2$ might then be called the \textbf{Fettuccine diffeology}. It was suggested by the participants of the AMS-EMS-SMF meeting 2022 in Grenoble that the case $r=3$ should be called the \textbf{Gnocchi diffeology}. For the rank-$r$-restricted diffeology the morphism $\theta_{k,M}$ of Axiom~(E1) is an isomorphism for all $k \leq r$ but not for $r < k < \dim M$.
\end{Example}

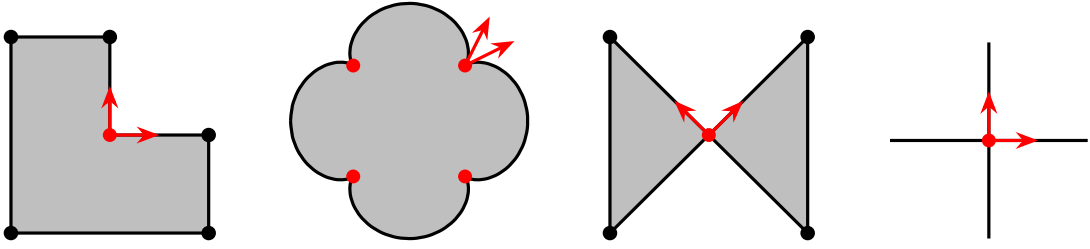
\begin{figure}
\centering
\begin{tikzpicture}[bolli]

\filldraw
    (0,0) node [point] {} -- 
    (2,0) node [point] {} -- 
    (2,1) node [point] {} -- 
    (1,1) node [redpoint] {} -- 
    (1,2) node [point] {} -- 
    (0,2) node [point] {} -- cycle;

\draw [->, red]
    (1, 1) -- (1, 1.5);

\draw [->, red]
    (1, 1) -- (1.5, 1);

\end{tikzpicture}
\qquad
\begin{tikzpicture}[bolli]

\filldraw[
    domain=-pi:pi,
    trig format=rad,
    samples=80,
    smooth]
    plot[parametric,
        mark indices={10,30,50,70}]
    (
        {cos(\x) + 0.2 * cos(5 * \x)},
        {sin(\x) + 0.2 * sin(5 * \x)} );

\draw
    ( {cos(180*10/40) + 0.2 * cos(5*180*10/40)},
      {sin(180*10/40) + 0.2 * sin(5*180*10/40)} ) 
      node [redpoint] {};
\draw      
    ( {-cos(180*10/40) - 0.2 * cos(5*180*10/40)},
      {sin(180*10/40) + 0.2 * sin(5*180*10/40)} ) 
      node [redpoint] {};
\draw      
    ( {cos(180*10/40) + 0.2 * cos(5*180*10/40)},
      {-sin(180*10/40) - 0.2 * sin(5*180*10/40)} ) 
      node [redpoint] {};
\draw      
    ( {-cos(180*10/40) - 0.2 * cos(5*180*10/40)},
      {-sin(180*10/40) - 0.2 * sin(5*180*10/40)} ) 
      node [redpoint] {};

\draw [->, red]
    ( {cos(180*10/40) + 0.2 * cos(5*180*10/40)},
      {sin(180*10/40) + 0.2 * sin(5*180*10/40)} ) 
    -- +(0.25, 0.5);

\draw [->, red]
    ( {cos(180*10/40) + 0.2 * cos(5*180*10/40)},
      {sin(180*10/40) + 0.2 * sin(5*180*10/40)} ) 
    -- +(0.5, 0.25);

\end{tikzpicture}
\qquad
\begin{tikzpicture}[bolli]

\filldraw
    (-1,1) -- 
    (0,0) node [redpoint] {} --
    (1,1) --
    (1,-1) --
    (0,0) --
    (-1,-1) -- cycle;

\draw
    (-1,-1) node [point] {}
    (1,-1) node [point] {}
    (-1,1) node [point] {}
    (1,1) node [point] {};

\draw [->, red]
    (0, 0) -- (0.35, 0.35);

\draw [->, red]
    (0, 0) -- (-0.35, 0.35);

\end{tikzpicture}
\qquad
\begin{tikzpicture}[bolli]

\draw (-1,0) -- (1,0) ;
\draw (0,-1) -- (0,1) ; 
\node[redpoint] at (0,0) {};

\draw [->, red]
    (0, 0) -- (0, 0.5);

\draw [->, red]
    (0, 0) -- (0.5, 0);

\end{tikzpicture}
\caption{Diffeological subspaces of $\bbR^2$ with non-elastic points marked in red, at which two tangent directions cannot be represented on the same plot.}
\label{fig:NonElastic}
\end{figure}

\subsection{The additional axioms}

So far we have the Axiom~(E1) that ensures that we have a fiberwise addtion on $TX$ and an inner derivative on differential forms. For the definition of the Lie bracket we need more structure. In particular, we need a natural morphism $\tau_X: T^2 X \to T^2 X$ that exchanges the order of differentiation when we apply the tangent functor twice. On a euclidean space $U \subset \bbR^n$, every tangent vector is represented by a path $\bbR \to U$ on some plot, so that a tangent vector on the manifold of tangent vectors is represented by a smooth path of smooth paths, which is the same as a smooth map $\bbR^2 \to U$. Exchanging the order of differentiation is achieved by exchanging the parameters, 
\begin{align*}
  \tau_{1 \leftrightarrow 2}: \bbR^2 
  &\longrightarrow \bbR^2
  \\
  (t_1, t_2) 
  &\longmapsto (t_2, t_1)
  \,,
\end{align*}
which descends by the commutative diagram
\begin{equation}
\label{eq:Intro04}
\begin{tikzcd}
C^\infty(\bbR^2, U)
\ar[r, "\tau_{1 \leftrightarrow 2}^*"] \ar[d]
&
C^\infty(\bbR^2, U)
\ar[d]
\\
T^2 U \ar[r, "\tau_U"']
&
T^2 U
\end{tikzcd}    
\end{equation}
to an endomorphism of $T^2 U$.

When we extend $\tau_U$ to diffeological spaces, the problem arises that the left Kan extension does not preserve the product of endofunctors, that is, the natural morphism
\begin{equation*}
  \theta^2_X: \Colim_{U \to X} T^2 U
  \longrightarrow
  T^2 X
\end{equation*}
is generally not an isomorphism. We could impose the condition that $\theta^2_X$ is an isomorphism, but this would be unnecessarily strong. It suffices to require the left Kan extension of $\tau_U$ to descend to a morphism $\tau_X: T^2 X \to T^2 X$. Since $\theta^2_X$ is a subduction for all $X$ (Proposition~\ref{prop:KanProdSubduction}), such a $\tau_X$ is unique. This condition can be expressed more intuitively in terms of the smooth families in the same way as for euclidean spaces. We can show that we can represent elements in $T^2 X$ by plots $\bbR^2 \to X$. More precisely, we have a subduction
\begin{equation*}
  \intHom(\bbR^2, X) \longrightarrow T^2 X
  \,,
\end{equation*}
where $\intHom$ denotes the inner hom of diffeological spaces, that is, the set of morphisms equipped with the functional diffeology. The second axiom can now be expressed in a way that is completely analogous to diagram~\eqref{eq:Intro04}.

\begin{AxiomE2}
There is a natural morphism $\tau_X: T^2 X \to T^2 X$, such that the diagram
\begin{equation*}
\begin{tikzcd}
\intHom(\bbR^2, X)
\ar[r, "\tau_{1 \leftrightarrow 2}^*"] \ar[d]
&
\intHom(\bbR^2, U)
\ar[d]
\\
T^2 X \ar[r, "\tau_X"']
&
T^2 X
\end{tikzcd}    
\end{equation*}
commutes.
\end{AxiomE2}

Next, consider the natural morphism $\lambda_X: TX \to T^2 X$ that maps $v \in TX$ to the vertical tangent vector on $TX$ represented by the path $t \mapsto tv$. On a smooth manifold, this morphism induces an isomorphism between every tangent space and the tangent space of the tangent space. For diffeologial vector spaces this can fail, as the following example shows.

\begin{Example}
\label{ex:CkDiffeologies}
Consider $\bbR^n$ equipped with $k$-times differentiable maps as plots. This is a diffeological vector space that we denote by $\bbR^n_{C^k}$. Its tangent diffeological space is given for $k > 0$ by
\begin{equation*}
  T\bbR^n_{C^k} \cong \bbR^n_{C^k} \times \bbR^n_{C^{k-1}}
  \,,
\end{equation*}
which shows that the vector space and its tangent fiber are not isomorphic. Assume that $k > 1$, so that we can apply the tangent functor twice. The vertical lift,
\begin{align*}
  \lambda_{\bbR^n_{\!C^k}}:
  \bbR^n_{C^k} \times \bbR^n_{C^{k-1}}
  &\longrightarrow
  \bbR^n_{C^k} \times \bbR^n_{C^{k-1}} 
  \times \bbR^n_{C^{k-1}} \times \bbR^n_{C^{k-2}}
  \\
  (x, v) &\longmapsto (x, 0, 0 , v)
  \,,
\end{align*}
is not a subduction.
\end{Example}

The definition of the Lie bracket in terms of the tangent structure yields a map from $X$ to the vertical subbundle of $T^2 X$ restricted to the zero section of $TX$. We have to be able to identify this bundle with $TX$ for the bracket to be again a vector field. This condition is not specific to diffeological spaces. A vector field on a $C^k$-manifold is a $C^k$-map. The commutator of two such vector fields is a $C^{k-1}$-map which is, therefore, not a vector field on the $C^k$-manifold. To exclude such phenomena we have to impose the following axiom:

\begin{AxiomE3}
The vertical lift $\lambda_X: TX \to T^2 X$ is an induction.
\end{AxiomE3}

There are two more axioms. For smooth manifolds the tangent functor commutes with pullbacks over submersions. This follows from the local standard form of submersions, which is proved using the implicit function theorem. Such a genuinely analytic result cannot hold for all diffeological spaces, which is why we need to impose the following axiom:

\begin{AxiomE4} 
The tangent functor commutes with fiber products of the tangent bundle, $T T_k X \cong T_k T X$.
\end{AxiomE4}

Finally, we want the diffeological spaces that satisfy our axioms to form a category. This requires the collection of diffeological spaces that satisfy the axioms to be closed under the functors $T_k$, which leads to the following axiom:

\begin{AxiomE5}
For every finite set of positive integers $k_1, \ldots, k_n$ the diffeological space $X' := T_{k_1} \cdots T_{k_n} X$ satisfies axioms (E1) through (E4).
\end{AxiomE5}

A diffeological space that satisfies Axioms (E1)-(E5) will be called \textbf{elastic}. If we drop Axiom (E5), then we still have a natural Cartan calculus on $X$. We call a diffeological space that satisfies Axioms (E1)-(E4) \textbf{weakly elastic}. The category of weakly elastic spaces is not closed under the functors $T_k$.

\begin{figure}
\centering
\begin{tikzpicture}[bolli]

\filldraw
    (0,0) rectangle (2,2);
    
\draw
    (0,0) node [point] {}
    (2,0) node [point] {}
    (0,2) node [point] {}
    (2,2) node [point] {};

\end{tikzpicture}
\qquad
\begin{tikzpicture}[bolli]

\filldraw[
    domain=-pi:pi,
    trig format=rad,
    samples=120,
    smooth]
    plot[parametric,
        mark=*,
        mark indices={5,25,45,65,85,105},
        mark size=1.5pt,
        mark options={color=black}]
    (
        {cos(\x) + 0.2 * sin(5 * \x)},
        {sin(\x) + 0.2 * cos(5 * \x)} );

\end{tikzpicture}
\qquad
\begin{tikzpicture}[bolli]

\filldraw
    (0,0) arc (180:0:1) -- 
    ++ (0,2) arc (360:180:1) -- cycle;

\draw
    (0,0) node [point] {}
    (2,0) node [point] {}
    (0,2) node [point] {}
    (2,2) node [point] {};

\end{tikzpicture}
\qquad
\begin{tikzpicture}[bolli]

\draw (0,0) -- (-1,0) ;
\draw (0,0) -- (0,1) ;
\draw (0,0) -- (1,-0.3) ;
\draw (0,0) -- (0.3,-1) ; 
\node[point] at (0,0) {};

\end{tikzpicture}

\caption{Elastic diffeological subspaces of $\bbR^2$. The tangent spaces are $0$ at the marked points, $\bbR$ at points on the black lines, and $\bbR^2$ at gray points in the interior.}
\label{fig:CornersCusps}
\end{figure}
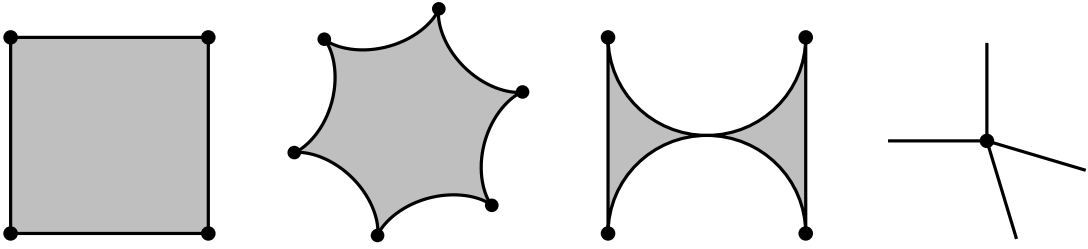

\subsection{Summary}

The main result about elastic diffeological spaces is the following: The left Kan extension of the tangent structure on euclidean spaces defines a tangent structure with scalar $\bbR$-mul\-ti\-pli\-ca\-tion on the category of elastic spaces (Theorem~\ref{thm:ElasticTanCat}). The proof of this statement, which is very long and technical, is carried out in detail in a much longer paper \cite{Blohmann:Elastic}. Here we will present an overview of the conceptual framework, the main properties, and important examples of elastic spaces.

In Section~\ref{sec:AbstractTangent} we review Rosick\'y's concept of abstract tangent structures. We also spell out the conditions for a compatible scalar multiplication, which is only mentioned briefly in the original paper \cite{Rosicky:1984}. For clarity, we spell out the tangent structure of euclidean spaces explicitly.

In Section~\ref{sec:LeftKanExtension} we state some new results about the left Kan extension from eulcidean spaces to diffeological spaces, that will be needed later. We use the definition of diffeological spaces as concrete sheaves on the site of euclidean spaces, which the best approach for the categorical constructions we will study. We then state, without proof, our results on the compatibility of the left Kan extension with products, coproducts, subductions, composition of endofunctors, and the $D$-topology of diffeological spaces, which are all needed for the proof of the main Theorem~\ref{thm:ElasticTanCat}.

Section~\ref{sec:Elastic} contains the formal definition of elastic spaces an the statement, without proof, of the main Theorem~\ref{thm:ElasticTanCat}. We discuss some alternative choices of axioms for which the theorem remains true. First, we observe that we can drop Axiom (E5) from the defintion of elastic spaces and still obtain a Cartan calculus on the space $X$. This category of weaker diffeological spaces will no longer be closed under the tangent functor and its fiber products, so that we no longer have a category with an abstract tangent structure in the sense of Rosick\'y. We can also slightly relax the Axiom (E1) of elasticity by requiring the the condition only holds for $k=2$. This weaker condition will be sufficient to prove Theorem~\ref{thm:ElasticTanCat}. As explained in the introduction, our choice of the stronger Axiom (E1) is motivated by our wish to obtain not just a tangent structure, but a full-fledged Cartan calculus. In an earlier version of the notion of elastic spaces presented in talks, Axioms~(E1), (E2), and~(E4) were replaced by a single condition, which later turned out to be unnecessarily strong. Finally, we state, without proof, that the category of elastic spaces is closed under restrictions to open subsets, coproducts, finite products, and retracts.

In Section~\ref{sec:Examples} we will give a number of examples for elastic diffeological spaces that show that, while the conditions of elasticity are quite strong, they still allow for an interesting range of applications. The first main result, Theorem~\ref{thm:ElasticGroups}, shows that a diffeological Lie group $G$ is elastic if and only if the natural map $\frakg \to T_0 \frakg$ from the vector space $\frakg = T_e G$ to its tangent space at $0$ is an induction. This is a surprisingly weak condition, which is not particular to diffeological spaces. For example, it is not satisfied by differentiable group structure on a $C^k$-manifold, $k < \infty$. The proof of Theorem~\ref{thm:ElasticGroups}, which is long and techincal, will be given in \cite{Blohmann:Elastic}. The second main result, Theorem~\ref{thm:FieldsElastic}, states that the diffeological space of sections of a smooth fiber bundle $F \to M$ is elastic. The proof of this theorem is again long and involved, so that it will be given in~\cite{Blohmann:Elastic}. More examples of elastic spaces in this sections are: manifolds with corners and cusps, diffeological vector spaces satisfying a mild extra condition, manifolds modelled on elastic vector spaces, mapping spaces $\intHom(X,A)$ to diffeological vector spaces satisfying $A \cong T_0 A$, the space of smooth $\bbR$-valued functions on a diffeological space, the mapping space of smooth manifolds.

\subsection{Outlook}

One of the motivations to develop the concept of elastic spaces came from classical field theory, where the basic strucure, the ``space'' of fields is the set of sections of a smooth fiber bundle $F \to M$ leaving it often unclear or implicit what ``space'' means mathematically. It is often observed that $\calF = \Gamma(M,F)$ is a Fr\'echet manifold, which is subsequently viewed as blanket license to treat $\calF$ as if it were an ordinary finite-dimensional smooth manifold. For example, it is often an implicit assumption that there is a natural differential bigraded algebra of smooth forms on $\calF \times M$ that restricts to the variational bicomplex on $J^\infty F$ \cite{DeligneFreed:1999}. In the same vein, in the study of symmetries, such as Noether's theorems or BV-theory, the constructions are often explained rigorously only for finite Lie groups acting on finite-dimensional manifolds and then generalized with a leap of faith to group-valued functions or diffeomorphisms acting on the spaces of fields. A closer analysis shows that only the diffeological structure that is being used. For example, a variation of a field is a smooth path of sections and a tangent vector to the space of solutions of the field equations (i.e.~a generalized Jacobi field) is represented by a smooth path of fields. All this suggests that there is a diffeological construction of a Cartan calculus on $\calF$. This approach is validated by Theorem~\ref{thm:FieldsElastic}.

Another motivation comes from geometric deformation theory. Conceptually, a deformation is a path in the moduli space of structures, such as the morphisms of an algebraic structure, riemannian metrics, or complex structures, all of which are equipped with a natural functional diffeology. This suggests that the geometric moduli spaces can be conceptualized by (higher) stacks that are presented by (higher) diffeological groupoids. The infinitesimal deformation theory should then be given by the fibers of the tangent bundle of the moduli space, which should be presented by the corresponding (higher) Lie algebroid. In order to define this procedure rigorously, a Lie theory for diffeological groupoids has to be developed. This has lead us to the conclusion that we need a tangent structure on the diffeological spaces of the groupoid. It is encouraging that for diffeological Lie groups the condition of elasticity is surprisingly weak. Lie theory for diffeological groupoids and their application to geometric deformation theory is work in progress.

\section{Abstract tangent structures}
\label{sec:AbstractTangent}

It is fairly straightforward to generalize the de Rham complex $\Omega(M)$, which is a contravariant functor from smooth manifolds to differential graded algebras, to other categories. While vector fields are in some sense dual to differential forms, their generalization is a much more difficult problem. An immediate obstacle is that the map $M \to \calX(M)$ that sends a manifold to its Lie algebra of vector fields is not a functor. It is the tangent bundle $TM \to M$ that is functorial in $M$ and, therefore, lends itself easily to generalizations. Given a generalized tangent bundle $\pi_X: TX \to X$ in some category, the vector fields are naturally defined as the sections of the morphism $\pi_X$. The obvious question is now the following:

\begin{Question}
Given a morphism $TX \to X$ in some catgory, what is the natural structure needed to equip its space of sections with the structure of a Lie algebra?
\end{Question}

This question turns out to be quite involved. A vector field on a smooth manifold can be identified with a derivation of its ring of smooth functions $C^\infty(M)$, which is closed under the commutator bracket of the ambient ring of endomorphisms of $C^\infty(M)$. However, $M \mapsto \Der(C^\infty(M))$ is still no functor. Moreover, in a generalized setting, the identification of sections of $TX \to X$ with derivations on some structure ring on $X$ seems to be an overly strong requirement that is extraneous to differential geometric considerations. So how can we define the Lie bracket of vector fields on $M$ directly in terms of the tangent functor using a categorical approach that lends itself to generalizations?

This issue has been solved by Rosick\'y in \cite{Rosicky:1984}. In a first step, we observe that on manifolds the vector space structure on vector fields is induced by the vector space structure on the tangent fibers. To allow for categories that do not contain the real numbers as object, we relax the structure of $\bbR$-vector space to that of an abelian group. The structure we then need on a generalized tangent bundle $TX \to X$ is the natural transformatios of addition $+_X: TX \times_X TX \to TX$ and a zero section $0_X: X \to TX$ that equip $TX \to X$ with the structure of an abelian group over $X$.

For the definition of the Lie bracket, we start from the coordinate formula
\begin{equation*}
  \Bigl[
  v^i \frac{\partial}{\partial x^i}
  w^j \frac{\partial}{\partial x^j}
  \Bigr]
  = \Bigl(v^i \frac{\partial w^j}{\partial x^i}
  - w^i \frac{\partial v^j}{\partial x^i} \Bigr)
    \frac{\partial}{\partial x^j}
\end{equation*}
for vector fields on $\bbR^n$. The right hand side is given by the derivation of $w$ with respect to $v$ minus the derivation of $v$ with respect to $w$. In order to generalize this formula we must make sense of the differentiation of one vector field with respect to another and the subtraction of the two terms.

The derivation of $w: X \to TX$ with respect to $v: X \to TX$ is given by the composition
\begin{equation*}
  X \xrightarrow{~v~} 
  TX \xrightarrow{~Tw~} 
  T^2 X
  \,.
\end{equation*}
However, we cannot yet subtract $Tw(v_x)$ and $Tv(w_x)$ since the basepoint of $Tw(v_x)$ in $\pi_{TX}: T(TX) \to TX$ is $w_x$, whereas that of $Tv(w_x)$ is $v_x$. We first have to exchange the order of differentiation of the twofold tangent bundle. That is, we need a natural transformation $\tau_X: T^2 X \to T^2 X$ that satisfies $\tau_X \circ \pi_{TX} = T\pi_X$. Then the basepoint of $\tau_X(Tv(w_x))$ is also $w_x$, so that we can take the difference
\begin{equation}
\label{eq:TanIntroBrack}
  Tw(v_x) - \tau_X\bigl( Tv(w_x) \bigr)
  = \bigl( +_{TX} \circ 
  ( Tw \circ v, - \tau_X \circ Tv \circ w) 
  \bigr)_x
  \,.
\end{equation}
The result lies in the vertical tangent bundle of $T^2 X \to TX$, that is, the kernel of $T\pi_{TX}$. In the last step we have to be able to identify at every point $w_x \in TX$ the vertical tangent space with $T_x X$ itself. 

In a smooth manifold, there is a morphism $\lambda_{2,M}: TM \times_M TM \to T^2M$ that maps the pair $(w_x, u_x)$ to the vertical tangent vector in $T^2 M$ that is represented by the path $t \mapsto w_x + t u_x$. This map induces an isomorphism $TM \times_M TM \cong \ker T\pi_{TM}$. In the generalized setting this structure is promoted to an axiom: We require a natural morphism $\lambda_{2,X}:  TX \times_X TX \to T^2 X$ that induces an isomorphism $TX \times_X TX \cong \ker T\pi_{TX}$. Using this isomorphism, the expression~\eqref{eq:TanIntroBrack} can be viewed as an element in $TX \times_X TX$ that can be projected onto the second factor which produces an element in $TX$, which is the value of the bracket $[v,w]$ at $x$.

For this bracket to satisfy the Jacobi identity, a number of compatibility relations and properties of the various structures, bundle projection, zero section, fiberwise addition, exchange of the order of differentiation, and the vertical lift have to be required \cite{Rosicky:1984}. Rosicky's axiomatization of all this structure is the basis of this paper.

\subsection{Preliminary remarks on terminology and notation}

\begin{Terminology}
\label{term:WibbleBundle}
Let Wibble be an algebraic theory. Let $X$ be an object in a category $\calC$ such that the overcategory $\calC \Comma X$ has all finite products (i.e.~pullbacks over $X$). A Wibble object in $\calC\Comma X$ will be called a \textbf{bundle of Wibbles over $X$}.
\end{Terminology}

In this paper Wibble will be one of: monoid, group, abelian group, module, $\bbR$-vector space (for categories containing $\bbR$ as an object). If $W \to X$ is a bundle of Wibbles and if the pullback $W_x = * \times_X W$ over a point $x: * \to X$ exists in $\calC$, then $W_x$ is a Wibble object in $\calC$. In other words, every fiber of a bundle of Wibbles is a Wibble object in $\calC$, which justifies the terminology. Note, that the notion of bundle of Wibbles does not make any assumptions on local trivializations, whatsoever. So a bundle of vector spaces over a manifold $M$ is more general than a vector bundle over $M$.

\begin{Remark}
The main purpose of Terminology~\ref{term:WibbleBundle} is to unify (for the purpose of this paper) the varied terminology found in the literature and to use a term that is self-explanatory for a category theorist. In \cite[p.~1]{Rosicky:1984} a bundle of (abelian) groups over an endofunctor $F: \calC \to \calC$ is called an ``natural (abelian) group bundle over $F$''. A bundle of vector spaces over a diffeological space $X$ is called a ``regular vector bundle'' in \cite{Vincent:2008}, a ``diffeological vector space over $X$'' in \cite{ChristensenWu:2016}, and a ``diffeological vector pseudo-bundle'' in \cite{Pervova:2016}.    
\end{Remark}

\begin{Notation}
\label{notation:Functors}
We will follow \cite{Rosicky:1984} for the notation of the compositions of functors and natural transformations. The composition of functors $G: \calA \to \calB$ and $F:\calB \to \calC$ will be denoted by juxtaposition $FG: \calA \to \calC$. Therefore, the horizontal composition of natural transformations $\alpha: F \to F'$ and $\beta: G \to G'$ (the Godement product) will also be denoted by juxtaposition $\alpha\beta: FG \to F'G'$. Its components are given by the following commutative diagram:
\begin{equation*}
\begin{tikzcd}[column sep=large]
FG(A) \ar[d, "\alpha_{G(A)}"'] \ar[r, "F(\beta_A)"] 
\ar[dr, "(\alpha\beta)_A"] & 
FG'(A) \ar[d, "\alpha_{G'(A)}"]
\\
F'G(A) \ar[r, "F'(\beta_A)"'] &
F'G'(A)
\end{tikzcd}    
\end{equation*}
The identity natural transformation $F \to F$ will be denoted by $F$, so that
\begin{align*}
  (F\beta)_A &= F(\beta_A)
  \\
  (\alpha G)_A &= \alpha_{G(A)}
  \,.
\end{align*}
The vertical composition of $\alpha$ with a natural transformation $\alpha': F' \to F''$ will be denoted by $\alpha' \circ \alpha: F \to F''$. Its components are given by $(\alpha' \circ \alpha)_A = \alpha'_A \circ \alpha_A$. The monoidal category of endofunctors will be denoted by $\End(\calC)$.
\end{Notation}

\subsection{Rosick\'{y}'s axioms}
\label{sec:RosickysAxioms}

In~\cite{Rosicky:1984}, Rosick\'{y} introduced the notion of \textbf{abstract tangent functor}, which captures much of the categorical structure of the tangent functor of manifolds. The following notion is implicit in Rosick\'y's definition:



\begin{Definition}
\label{def:SymStruc}
Let $F: \calC \to \calC$ be a functor and $\tau: F^2 \to F^2$ a natural transformation. Let $\tau_{12} := \tau\,F$ and $\tau_{23}:= F\,\tau$ be the two trivial extensions of $\tau$ to natural transformations $F^3 \to F^3$. We call $\tau$ a \textbf{braiding on $F$} if it satisfies the braid relations $\tau_{12} \circ \tau_{23} \circ \tau_{12} = \tau_{23} \circ \tau_{12} \circ \tau_{23}$. A braiding $\tau$ is called a \textbf{symmetric structure on $F$} if it satisfies $\tau\circ \tau = F^2$.
\end{Definition}

\begin{Remark}
A symmetric structure on $F$ defines an action of the symmetric group $S_n$ on $F^n$.
\end{Remark}

A bundle of groups over $X$ consists of a morphism $\pi: A \to X$, the bundle projection, together with the morphisms $0: X \to A$ and $+: A \times_X A \to A$ of the group structure. Let $\pi': A' \to X'$, $0': X' \to A'$, $+': A' \times_{X'} A' \to A'$ be another bundle of groups. A \textbf{morphism of bundles} is a commutative diagram
\begin{equation}
\label{eq:MorphBundles1}
\begin{tikzcd}
A \ar[r, "\phi"] \ar[d, "\pi"'] & A' \ar[d, "\pi'"]
\\
X \ar[r, "\psi"'] & X' 
\end{tikzcd}    
\end{equation}
There is a unique morphism $\phi \times_\psi \phi: A \times_X A \to A' \times_{X'} A'$ that makes the following diagram commutative:
\begin{equation}
\label{eq:MorphBundles2}
\begin{tikzcd}
A \times_X A \ar[r, "\phi \times_\psi \phi" ] \ar[d] 
&
A \times_X A \ar[d]
\\
A \times A \ar[r, "\phi \times \phi"']
& 
A' \times A'
\end{tikzcd}  
\end{equation}
The pair $(\phi,\psi)$ is a \textbf{morphism of bundles of groups} if the diagram
\begin{equation}
\label{eq:MorphBundles3}
\begin{tikzcd}
A \times_X A \ar[r, "\phi \times_\psi \phi"] 
\ar[d, "+"'] & A' \times_X A' \ar[d, "+'"]
\\
A \ar[r, "\phi"'] & A'
\end{tikzcd}    
\end{equation}
commutes. An endofunctor $F$ \textbf{preserves the fiber product} if the natural morphism of bundles over $FX$,
\begin{equation}
\label{eq:MorphBundles4}
  \nu_{k,X}:
  F(A \times_X^{\pi,\pi} \ldots \times_X^{\pi,\pi} A) \longrightarrow
  FA \times_{FX}^{F\pi, F\pi} \ldots \times_{FX}^{F\pi, F\pi} FA
  \,,
\end{equation}
where both sides have the same number $k$ of factors, is an isomorphism for all $k$.

\begin{Definition}[Sec.~2 in\cite{Rosicky:1984}, Def.~2.3 in  \cite{CockettCruttwell:2014}]
\label{def:TangentStructure}
A \textbf{tangent structure} of a category $\calC$ consists of a functor $T: \calC \to \calC$ together with natural transformations $\pi: T \to \Id$, $0: 1 \to T$, $+: T_2 \to T$, $\lambda: T \to T^2$, and $\tau: T^2 \to T^2$, such that the following axioms hold:
\begin{itemize}

\item \textbf{Fiber products:} 
The pullbacks
\begin{equation}
\label{eq:TanFun1}
  T_k := \underbrace{T \times_\Id T \times_\Id \ldots \times_\Id T}_{k \text{ factors}}
\end{equation}
over $T \stackrel{\pi}{\to} 1$ exist for all $k \geq 1$, are pointwise, and preserved by $T$.

\item \textbf{Bundle of abelian groups:}
$T \stackrel{\pi}{\to} \Id$ with neutral element $0$ and addition $+$ is a bundle of abelian groups over $1$ (Terminology~\ref{term:WibbleBundle}).

\item \textbf{Symmetric structure:}
$\tau: T^2 \to T^2$ is a symmetric structure on $T$ (Definition~\ref{def:SymStruc}). Moreover, $\tau$ is a morphism of bundles of groups. That is, the diagrams
\begin{equation}
\label{eq:TanFun2}
\begin{tikzcd}[column sep=1em]
T^2 \ar[rr, "\tau"] \ar[dr, "T \pi"'] && T^2 \ar[dl, "\pi T"]
\\
& T &
\end{tikzcd}    
\end{equation}
and
\begin{equation}
\label{eq:TanFun2b}
\begin{tikzcd}
T T_2 \ar[r, "\nu_2"] \ar[d, "T+"']
&
T^2 \times_T^{T\pi, T\pi} T^2  
\ar[r, "\tau \times_T \tau"] 
&
T^2 T
\ar[d, "+T"] 
\\
T^2 \ar[rr, "\tau"'] 
&& 
T^2
\end{tikzcd}    
\end{equation}
commute, where $\nu_2$ is morphism~\eqref{eq:MorphBundles4} for $A=TX \stackrel{\pi}{\to} X$, $F=T$, and $k=2$.

\item \textbf{Vertical lift:} 
The diagrams
\begin{equation}
\label{eq:TanFun3}
\begin{tikzcd}
T \ar[r, "\lambda"] \ar[d, "\pi"'] & T^2 \ar[d, "\pi T"]
\\
\Id \ar[r, "0"'] & T
\end{tikzcd}    
\qquad\qquad
\begin{tikzcd}
T \ar[r, "\lambda"] \ar[d, "\lambda"'] & 
T^2 \ar[d, "\lambda T"]
\\
T^2 \ar[r, "T \lambda"'] & T^3
\end{tikzcd}    
\end{equation}
commute. Moreover, the first diagram is a morphism of bundles of groups, that is $(+T) \circ (\lambda \times_0 \lambda) = \lambda \circ +$.

\item \textbf{Compatibility of vertical lift and symmetric structure:}
The diagrams
\begin{equation}
\label{eq:TanFun5}
\begin{tikzcd}[column sep=tiny]
& T \ar[dl, "\lambda"'] \ar[dr, "\lambda"] &
\\
T^2 \ar[rr, "\tau"'] & &
T^2
\end{tikzcd}    
\qquad\qquad
\begin{tikzcd}
T^2 \ar[r, "T\lambda"] \ar[d, "\tau"'] & 
T^3 \ar[r, "\tau T"] &
T^3 \ar[d, "T\tau"]
\\
T^2 \ar[rr, "\lambda T"'] & &
T^3
\end{tikzcd}    
\end{equation}
commute.

\item \textbf{The vertical lift is a kernel:} The diagram
\begin{equation}
\label{eq:TanFun4}
\begin{tikzcd}
T \ar[r, "\lambda"] \ar[d, "\pi"'] & T^2 \ar[d, "{(\pi T, T\pi)}"]
\\
\Id \ar[r, "{0 \times_1 0}"'] & T_2
\end{tikzcd}    
\end{equation}
is a pointwise pullback. 

\end{itemize}
\end{Definition}

\begin{Terminology}
In \cite{CockettCruttwell:2014} and subsequent work, Rosick\'{y}'s original condition that $T \to \Id$ be a bundle of abelian groups was relaxed to a bundle of abelian monoids. In this terminology, Rosicky's stronger notion is called a tangent structure with negatives. All tangent structures in this paper will be with negatives.
\end{Terminology}

\begin{Remark}
Diagram~\eqref{eq:TanFun4} is a pointwise pullback if and only if
\begin{equation*}
\begin{tikzcd}[column sep=large]
T \ar[r, "\lambda"] &
T^2 
\ar[r, "\pi T", shift left=4] 
\ar[r, "T\pi", shift left=0] 
\ar[r, "0 \circ \pi \circ \pi T"', shift left=-3] 
&
T
\end{tikzcd}
\end{equation*}
is a pointwise triple equalizer. This condition is the original axiom in \cite{Rosicky:1984}. 
\end{Remark}

\begin{Remark}
The vertical lift can be extended by the additive bundle structure to the map
\begin{equation}
\label{eq:VertLiftExt}
  \lambda_2:
  T_2 \xrightarrow{~T0 \times_0 \lambda~}
  T_2 T
  \xrightarrow{~+T~}
  T^2
  \xrightarrow{~\tau~}
  T^2
  \,.
\end{equation}
It was shown in \cite[Lem.~3.10]{CockettCruttwell:2014}, assuming all other axioms of a tangent structure (with negatives), that axiom~\eqref{eq:TanFun4} is satisfied if and only if
\begin{equation*}
\begin{tikzcd}
T_2 \ar[r, "\lambda_2"] \ar[d, "\pi \circ \pr_1"'] & 
T^2 \ar[d, "{T\pi}"]
\\
\Id \ar[r, "0"'] & T
\end{tikzcd}    
\end{equation*}
is a pointwise pullback.

\end{Remark}

\subsection{Scalar multiplication}

Let $R$ be a ring object in the category $\calC$. This gives rise to an endofunctor $R \times \Id: \calC \to \calC$, $X \mapsto R \times X$, which is equipped with the projection $\pr_\Id: R \times \Id \to \Id$. The ring structure of $R$ equips $R \times \Id \to \Id$ with the structure of a ring internal to endofunctors over $\Id$. Let $\pi: T \to \Id$ be an abelian group object in the category of endofunctors over $\Id$. An $(R \times \Id \to \Id)$-module structure on $T \to \Id$ is given explicitly by a natural transformation
\begin{equation}
  \kappa_X: R \times TX 
  \longrightarrow TX
  \,,
\end{equation} 
such that the following diagrams commute for all $X \in \calC$:
\begin{itemize}

\item[(i)] Morphism of bundles:
\begin{equation*}
\begin{tikzcd}[column sep={tiny}]
R \times TX \ar[rr, "\kappa_X"] \ar[dr, "\pi_X \circ \pr_2"'] &&
TX \ar[dl, "\pi_X"]
\\
& X &
\end{tikzcd}
\end{equation*}

\item[(ii)] Associativity:
\begin{equation*}
\begin{tikzcd}[column sep={large}]
R \times R \times TX \ar[r, "\id_R \times \kappa_X"] 
\ar[d, "\cdot \times \id_{\bbR \times TX}"'] &
R \times TX \ar[d, "\kappa_X"]
\\
R \times TX \ar[r, "\kappa_X"'] & TX
\end{tikzcd}
\end{equation*}

\item[(iii)] Unitality:
\begin{equation*}
\begin{tikzcd}
\{1\} \times TX \ar[r, hook] \ar[dr, "\cong"'] & 
R \times TX \ar[d, "\kappa_X"]
\\
& TX
\end{tikzcd}    
\end{equation*}

\item[(iv)] Linearity in $R$:
\begin{equation*}
\begin{tikzcd}[column sep={large}]
R \times R \times TX
\ar[r, "+ \times \id_{TX}"] 
\ar[d, "\id_{R^2} \times \Delta_{TX}"'] &
R \times TX \ar[dd, "\kappa_X"]
\\
(R \times TX) \times_X (R \times TX)
\ar[d, "\kappa_X \times_X \kappa_X"']
&
\\
TX \times_X TX 
\ar[r, "+_X"'] 
&
TX
\end{tikzcd}
\end{equation*}
Here $\Delta_{TX}:  TX \to TX \times_X TX$ is the diagonal morphism and the factors of the codomain are reordered.

\item[(v)] Linearity in $TX$:
\begin{equation*}
\begin{tikzcd}[column sep={large}]
R \times TX \times_X TX
\ar[r, "\id \times +_X"] 
\ar[d, "\Delta_R \times \id_{T_2 X}"'] &
R \times TX \ar[dd, "\kappa_X"]
\\
(R \times TX) \times_X (R \times TX)
\ar[d, "\kappa_X \times_X \kappa_X"']
&
\\
TX \times_X TX 
\ar[r, "+_X"'] 
&
TX
\end{tikzcd}
\end{equation*}
Here $\Delta_R:  R \to R \times R$ is the diagonal morphisms and the factors of the codomain are reordered.
\end{itemize}

We will call this structure more succinctly an \textbf{$R$-module structure} on $T \to \Id$ and $T \to \Id$ a \textbf{bundle of $R$-modules} (Terminology~\ref{term:WibbleBundle}). If $T$ is part of a tangent structure on $\calC$, then we also have to require the compatibility with the symmetric structure and the vertical lift.

\begin{Definition}
\label{def:RScalar}
Let $R$ be a ring internal to a category $\calC$ with a tangent structure. An $R$-module structure $\kappa_X: R \times TX \to TX$ will be called a \textbf{scalar multiplication} of the tangent structure if the following diagrams commute for all $X \in \calC$:
\begin{itemize}

\item[(vi)] Compatibility with the symmetric structure:
\begin{equation*}
\begin{tikzcd}[column sep={tiny}]
R \times T^2 X \ar[rr, "\id_R \times \tau_X"] \ar[dr, "\kappa_{TX}"'] &&
R \times T^2 X \ar[dl, "T\kappa_X"]
\\
& T^2 X &
\end{tikzcd}
\end{equation*}

\item[(vii)] Compatibilty with the vertical lift:
\begin{equation*}
\begin{tikzcd}[column sep={large}]
R \times TX 
\ar[r, "\id_R \times \lambda_X"] 
\ar[d, "\kappa_X"'] &
R \times T^2 X \ar[d, "\kappa_{TX}"]
\\
TX 
\ar[r, "\lambda_X"'] 
&
T^2 X
\end{tikzcd}
\end{equation*}

\end{itemize}
\end{Definition}

The tangent structures we consider here will all be equipped with an $\bbR$-scalar multiplication.

\begin{Remark}
As is the case for any module structure, the commutative diagram~(v) implies that the scalar multiplication by $0 \in R$ sends $TX$ to the zero section, that is, the diagram
\begin{equation*}
\begin{tikzcd}
\{0\} \times TX \ar[r, hook] \ar[d, "\pi_X \circ \pr_2"'] & 
R \times TX \ar[d, "\kappa_X"]
\\
X \ar[r, "0_X"'] & TX
\end{tikzcd}    
\end{equation*}
is commutative. If $\kappa$ is such that this diagram and diagrams (i)-(iii) are commutative, then $\kappa$ will be called an \textbf{$R$-cone} structure and $T \to 1$ a \textbf{bundle of $R$-cones}.
\end{Remark}

\subsection{The Lie bracket of vector fields}
\label{sec:LieBracket}

\begin{Definition}
Let $\calC$ be a category with a tangent structure. A \textbf{vector field} on $X \in \calC$ is a section of $\pi_X: TX \to X$.
\end{Definition}

The bracket of two vector fields $v,w: X \to TX$ is defined as follows. The composition of $v$ and $Tw: TX \to T^2 X$ satisfies
\begin{equation*}
\begin{split}
  \pi_{TX} \circ Tw \circ v 
  &= w \circ \pi_X \circ v
  = w \circ \id_X
  \\
  &= w
  \,.
\end{split}
\end{equation*}
When we exchange $v$ and $w$, we have $\pi_X \circ Tv \circ w = v$. In order to be able to subtract the two terms in the fiber product $T^2 X \times_{TX} T^2 X$, we have two apply the symmetric structure on $T^2$, so that we obtain
\begin{equation*}
\begin{split}
  \pi_{TX} \circ \tau_X \circ Tv \circ w 
  &= T\pi_X \circ Tv \circ w
  = T\id_X \circ w
  = \id_{TX} \circ w
  \\
  &= w
  \,.
\end{split}
\end{equation*}
This shows that $Tw \circ v$ and $\tau_X \circ Tv \circ w$ project to the same fiber of $\pi_{TX}: T^2 X \to TX$, so that we can take the difference
\begin{equation*}
  \delta(v,w)(x)
  := 
  (Tw \circ v)(x) - (\tau_X \circ Tv \circ w)(x)
  \,,
\end{equation*}
where the minus denotes the difference in the bundle of abelian groups $\pi_{TX}: T^2 X \to TX$. We have
\begin{equation*}
  \pi_{TX} \circ \delta(v,w) = 0 = T\pi_X \circ \delta(v,w)
  \,,
\end{equation*}
so that the map $\delta(v,w): X \to T^2 X$ takes values in the kernel of $T\pi_X: T^2 X \to TX$, which is isomorphic to $TX \times_X TX$. By projecting on the second factor we thus obtain the vector field $[v,w]: X \to TX$. This construction can be summarized by the following commutative diagram:
\begin{equation}
\label{eq:BracketDef}
\begin{tikzcd}
X 
\ar[r, "\Delta_X"] 
\ar[d, "{[v,w]}"']
\ar[dd, "{\exists!}", dashed, bend left=45]
\ar[ddrr, "\exists!", dashed]
&
X \times X \ar[r, "v \times w"]
&
TX \times TX \ar[r, "Tw \times Tv"]
&
T^2 X \times T^2 X
\ar[dd, "\id_X \times \tau_X"]
\\
TX &&&
\\
TX \times_X TX
\ar[dr, phantom, "\lrcorner", very near start]
\ar[u, "\pr_2"]
\ar[r, "\lambda_{2,X}"] 
\ar[d, "\pi_X"']
&
T^2 X 
\ar[d, "{T\pi_X}"]
&
T^2 X \times_{TX} T^2 X
\ar[dr, phantom, "\lrcorner", very near start]
\ar[l, "-_{TX}"']
\ar[r, hookrightarrow]
\ar[d]
&
T^2 X \times T^2 X
\ar[d, "\pi_{TX} \times \pi_{TX}"]
\\
X
\ar[r, "{0_X}"']
&
TX
&
TX
\ar[r, "\Delta_{TX}"']
&
TX \times TX
\end{tikzcd}
\end{equation}
This shows that all of the tangent structure is needed for the definition of the bracket of vector fields. It was announced in \cite{Rosicky:1984} and proved in \cite{CockettCruttwell:2015} with the input of Rosick\'y that $[v,w]$ satisfies the Jacobi relation.

The set of vector fields $\Gamma(X, TX)$ has the structure of an abelian group with addition 
\begin{equation*}
  v + w := +_X \circ (v \times w) \circ \Delta_X
  \,,
\end{equation*}
where $\Delta_X: X \to X \times X$ is the diagonal morphism. When the tangent structure has a scalar multiplication by $R$, then $\Gamma(X, TX)$ is a module over the ring $\calC(X,R)$ of $R$-valued functions, given by 
\begin{equation*}
  \kappa(f,v) = \kappa_X \circ (f \times v) \circ \Delta_X
  \,.
\end{equation*}

\subsection{The tangent structure of euclidean spaces}
\label{sec:TangentEuclidean}

The eponymous example for tangent structures is the tangent functor of open subsets of real vector spaces, which is the local model for the tangent functor of smooth manifolds. Let $\Eucl$ denote the category which has open subsets of $\bbR^n$, $n\geq 0$ as objects and smooth maps as morphisms. $\Eucl$ will be called the category of \textbf{euclidean spaces}. Its tangent functor will be denoted by
\begin{equation*}
  \eFun{T}: \Eucl \longrightarrow \Eucl
  \,.
\end{equation*}

On an open subset $U \subset \bbR^n$, the functors that appear in the definition~\ref{def:TangentStructure} of a tangent category are given explicitly by
\begin{equation*}
\begin{aligned}
  \eFun{T} U &= U \times \bbR^n
  \\
  \eFun{T}^2 U &= U \times \bbR^n \times \bbR^n \times \bbR^n
  \\
  \eFun{T}^k U &= U \times (\bbR^n)^{2^k-1}
  \\
  \eFun{T}_2 U &= U \times \bbR^n \times \bbR^n
  \\
  \eFun{T}_k U &= U \times (\bbR^n)^k
  \,.
\end{aligned}
\end{equation*}
On a smooth map $f: U \to V \subset \bbR^m$ the functors are given by
\begin{align}
  \eFun{T} f: (u,u_0^i) &\longmapsto
  \Bigl( f(u), \frac{\partial f^a}{\partial x^i} u_0^i \Bigr)
  \label{eq:T1f}\\
  \eFun{T}^2 f: (u, u_0^i, u_1^i, u_{01}^i) &\longmapsto
  \Bigl( f(u), 
  \frac{\partial f^a}{\partial x^i} u_0^i , 
  \frac{\partial f^a}{\partial x^i} u_1^i ,
  \frac{\partial f^a}{\partial x^i} u_{01}^i +
  \frac{\partial^2 f^a}{\partial x^i \partial x^j} u_0^i u_1^j 
  \Bigr)
  \label{eq:T2f}\\
  \eFun{T}_2 f: (u, u_0^i, u_1^i) &\longmapsto
  \Bigl( 
  f(u), 
  \frac{\partial f^a}{\partial x^i} u_0^i ,
  \frac{\partial f^a}{\partial x^i} u_1^i 
  \Bigr)
  \,. \notag
\end{align}
The formulas for $\eFun{T}^k$ and $\eFun{T}_k$ are analogous. The natural transformations of the tangent category structure are given by
\begin{align*}
  \eFun{\pi}_U : (u,u_0) &\longmapsto u
  \\
  \eFun{0}_U : u &\longmapsto (u, 0)
  \\
  \eFun{+}_U : (u, u_0, v_0) &\longmapsto (u, u_0 + v_0)
  \\
  \eFun{\lambda}_U : (u,u_0) &\longmapsto (u,0,0,u_0)
  \\
  \eFun{\tau}_U : (u, u_0, u_1, u_{01} ) &\longmapsto (u, u_1, u_0, u_{01} )
  \,.
\end{align*}
The commutativity of $\eFun{T}_2$ and $\eFun{T}$ is given by the isomorphism
\begin{equation*}
\begin{aligned}
  \eFun{T}(\eFun{T}_2 U) &\longrightarrow \eFun{T}_2(\eFun{T} U)
  \\
  \bigl( (u, u_0, v_0), (u_1, u_{01}, v_{01}) \bigr)
  &\longmapsto
  \bigl( ( u, u_1), (u_0, u_{01}), (v_0, v_{01}) \bigr)
\end{aligned}
\end{equation*}
The bundle projection extends to $\eFun{T}^2$ as
\begin{align*}
  (\eFun{\pi}\eFun{T})_U = \eFun{\pi}_{TU}: (u, u_0, u_1, u_{01}) &\longmapsto (u, u_0)
  \\
  (\eFun{T}\eFun{\pi})_U = \eFun{T}\eFun{\pi}_U: (u, u_0, u_1, u_{01}) &\longmapsto (u, u_1)
  \,.
\end{align*}
The other natural transformations that appear in the definition, $\eFun{+}\eFun{T}$, $\eFun{T}\eFun{+}$, $\eFun{\lambda}\eFun{T}$, $\eFun{T}\eFun{\lambda}$, $\eFun{\tau}\eFun{T}$, and $\eFun{T}\eFun{\tau}$, are obtained in a similar way. The extension~\eqref{eq:VertLiftExt} of the vertical lift is given by
\begin{equation*}
  \eFun{\lambda}_2: (u, u_0, v_0) \longmapsto (u, u_0, 0,  v_0)
  \,.
\end{equation*}
The following propositions can be checked by explicit elementary calculation:

\begin{Proposition}
\label{prop:EuclTangent}
$\Eucl$ with the tangent functor $\eFun{T}$ and the natural transformations $\eFun{\pi}$, $\eFun{0}$, $\eFun{+}$, $\eFun{\tau}$, and $\eFun{\lambda}$ is a tangent structure on $\Eucl$.
\end{Proposition}

\begin{Proposition}
The fiberwise multiplication by real numbers, 
\begin{equation*}
\begin{aligned}
  \eFun{\kappa}_U: \bbR \times \eFun{T} U 
  &\longrightarrow \eFun{T} U
  \\
  \bigl(r, (u,u_0) \bigr)
  &\longmapsto (u, r u_0)
  \,,
\end{aligned}
\end{equation*}
is a scalar multiplication of the tangent structure (Definition~\ref{def:RScalar}).
\end{Proposition}

\section{Left Kan extension to diffeological spaces}
\label{sec:LeftKanExtension}

The structures on smooth manifolds that we wish to generalize to diffeological spaces, such as the tangent bundle and the algebra of differential forms, are local and universal in the sense that they are defined on all open subsets of $\bbR^n$ and then glued together along an atlas. In categorial terms, the local structure is given by a functor $F: \Eucl \to \calC$ and the glueing operation by the colimit
\begin{equation*}
  FM := \Colim_{U \to M} FU
\end{equation*}
over a maximal atlas. In categorical terms, $FM$ is the pointwise left Kan extension of $F$ along the inclusion of euclidean spaces into smooth manifolds. For the generalization of this construction to diffeological spaces we replace the charts of the smooth manifold $M$ with the plots of the diffeological space $X$,
\begin{equation*}
  FX := \Colim_{U \to X} FU
  \,.
\end{equation*}
This is the pointwise left Kan extension of $F$ along the inclusion of euclidean spaces into diffeological spaces.

In order to use the left Kan extension as universal tool to generalize differential geometric structures to diffeological spaces, we have to study its categorical properties in some detail. For this it is best to work with the definition of diffeological spaces as concrete sheaves on euclidean spaces, which was elucidated in \cite{BaezHoffnung:2011}. The site $\Eucl$ of euclidean spaces is concrete, which means that there is a faithful functor $\Eucl \to \Set$, $U \to |U|$ that maps covers to surjective maps. A sheaf $D: \Eucl^\op \to \Set$ is concrete if it is a subsheaf of $U \mapsto \Set(|U|, D(*))$. Explicitly, this means that $D(U)$ is a collection of maps of sets $|U| \to X := D(*)$, called plots, that satisfy the properties of a sheaf. In this way, we recover the original definition of diffeology.

As is the case for any category of concrete sheaves, the category of diffeological spaces $\Dflg$ is a quasi-topos, a category with a classifier for strong subobjects, which has a number of good properties (Proposition~\ref{prop:ConcSheafProperties}). Other good properties of $\Dflg$ are inherited from the site $\Eucl$. For example, the category of plots of a diffeological space, the index category used to compute the pointwise left Kan extension, is sifted, which suggest a compatibility with finite products.

Every representable presheaf is a concrete sheaf, so that the Yoneda embedding factors through a full and faithful functor 
\begin{equation*}
  y: \Eucl \longrightarrow \Dflg
  \,.
\end{equation*}
While $y$ is simply the Yoneda embedding restricted on its codomain, the restriction changes the computation of colimits. A colimit in $\Dflg$ is given by first taking the pointwise colimit in the functor category $\Set^{\Eucl^\op}$ and then applying the left adjoint of the inclusion $\Dflg \to \Set^{\Eucl^\op}$. This second step has consequences for the properties of the left Kan extension, which is given by a colimit.

Most structures that we will consider are given by endofunctors $F: \Eucl \to \Eucl$. Since $\Eucl$ is not cocomplete and since we want to obtain an endofunctor of diffeological spaces, we have to place the codomain of $F$ in diffeological spaces by composing with $y$ before taking the Kan extension. The Kan extended endofunctor will be denoted by
\begin{equation*}
  \Lanyy F := \Lan_y yF : \Dflg \longrightarrow \Dflg
  \,.
\end{equation*}
The main question we will address in this section is the following.

\begin{Question}
What are the categorical properties of $\Lanyy: \End(\Eucl) \to \End(\Eucl)$ and which properties of $F$ are preserved by $\Lanyy$?
\end{Question}

It follows from the naturality of the Kan extension that $\Lanyy$ is a functor 
\begin{equation*}
  \Lanyy: \End(\Eucl) \longrightarrow \End(\Dflg)
\end{equation*}
of categories of endofunctors, which means that the vertical composition of natural transformations and, therefore, commutative diagrams of natural transformations are preserved. This is the most important property for our purposes, since it implies that the commutative diagrams that appear in the Definition~\ref{def:TangentStructure} of tangent structures carry over to the diffeological setting.

Unfortunately, this is where the good news about the left Kan extension to diffeological spaces end. Unlike the Kan extension to presheaves, the Kan extension of a functor to concrete sheaves does generally not preserve colimits, not even finite coproducts. Nor does $\Lanyy$ preserve the composition of endofunctors. Fortunately, the functors $F$ we want to extend have some good categorical properties, that entail good properties of their Kan extensions $\Lanyy F$.

Using that the category of plots $U \to X$ is sifted, we show that if $F$ preserves finite products, then so does its Kan extension $\Lanyy F$ (Proposition~\ref{prop:FpreservesProducts}). It also implies that $\Lanyy$ preserves finite products of endofunctors (Proposition~\ref{prop:LanyyProdFunc}). For the compatibility with coproducts we have to assume that $F$ is a cosheaf, that is, $F: \Eucl^\op \to \Eucl^\op$ is a sheaf. Then $\Lanyy F$ preserves coproducts (Corollary~\ref{cor:KanCoproducts2}) and subductions (Proposition~\ref{prop:KanStrongEpis}).

While $\Lanyy$ does not preserve the composition of endofunctors, there is a natural morphism
\begin{equation}
\label{eq:ProdMor1}
  \Lanyy(FG) \longrightarrow (\Lanyy F)(\Lanyy G)
  \,,
\end{equation}
which is generally not an isomorphism. We can show that if $F$ is a cosheaf, then \eqref{eq:ProdMor1} applied to a diffeological space is a subduction (Proposition~\ref{prop:KanProdSubduction}). The morphism 
\begin{equation*}
  \Lanyy(FGH) \longrightarrow \Lanyy(F)\Lanyy(G)\Lanyy(H)   
\end{equation*}
we obtain by applying \eqref{eq:ProdMor1} twice does not depend on whether we first apply it to $(FG)H$ or to $F(GH)$ (Proposition~\ref{prop:KanAssoc}). In this sense $\mu$ is associative. It is straightforward to see that \eqref{eq:ProdMor1} is natural in $F$ and $G$ (Proposition~\ref{prop:KanHorComp}).

Finally, we study the compatibility of the Kan extension with the $D$-topology of diffeological spaces. We observe that many of the functors on euclidean spaces we are interested in come with a natural transformation $F \to \Id$ such that the pullback along any open embedding $V \to U$ satisfies $FV \cong V \times_U FU$. We call such an $F \to \Id$ a local bundle, by default of a better term. We can show that if $F \to \Id$ is local, then so is its Kan extension $\Lanyy F \to \Id$ (Proposition~\ref{prop:KanLocal}). Then we prove that a morphism $FX \to GX$ of local bundles is an induction, subduction, epimorphism, or isomorphism if all restrictions to the open subsets of a cover of $X$ are (Proposition~\ref{prop:LocFuncCover}).

\subsection{Diffeological spaces as concrete sheaves}
\label{sec:DflgSpaces}

Recall that $\Eucl$ denotes the category which has all open subsets of euclidean spaces $\bbR^n$, $n\geq 0$ as objects and all smooth maps as morphisms. Open covers define a Grothendieck pretopology. 

\begin{Definition}
The small category $\Eucl$ together with the Grothendieck topology generated by the pretopology of open covers will be called the \textbf{site of euclidean spaces}.
\end{Definition}

The terminal object in $\Eucl$ is $* := \bbR^0$. The functor of points,
\begin{equation*}
\begin{aligned}
  | \Empty | : \Eucl &\longrightarrow \Set \\
  U &\longmapsto \Eucl(*,U)
  \,,
\end{aligned}    
\end{equation*}
is faithful, so that it equips $\Eucl$ with the structure of a concrete category. Moreover, every cover $\{U_i \to U\}$ is surjective on the underlying sets. A site with these properties is called \textbf{concrete}.

Let $X: \Eucl^\op \to \Set$ be a presheaf. Then there is a morphism of presheaves defined by
\begin{equation}
\label{eq:ConcetePrsh1}
\begin{aligned}
  X(U) &\longrightarrow \Set\bigl(|U|, |X| \bigr)
  \\  
  p &\longmapsto
  \bigl((* \stackrel{u}{\to} U) \mapsto X_u(p) \bigr)
  \,,
\end{aligned}
\end{equation}
where $X_u \equiv X(u): X(U) \to X(*)$ is the restriction of $X(U)$ to the point $u$.

\begin{Definition}
\label{def:ConreteSheaf}
A presheaf $X: \Eucl^\op \to \Set$ on a concrete site is 
\textbf{concrete} if \eqref{eq:ConcetePrsh1} is a monomorphism, i.e.~if the maps defined in~\eqref{eq:ConcetePrsh1} are injective for all $U \in \Eucl$. A sheaf is concrete if it is concrete as a presheaf. A morphism between concrete sheaves is a morphism of the underlying presheaves.
\end{Definition}

\begin{Definition}
\label{def:Difflg}
A concrete sheaf on the site of euclidean spaces is called a \textbf{diffeological space}. A morphism of diffeological spaces is a morphism of sheaves. The category of diffeological spaces will be denoted by $\Dflg$.
\end{Definition}

\begin{Theorem}[Thm.~5.25 in \cite{BaezHoffnung:2011}]
\label{thm:ConcQuasiCat}
The category of diffeological spaces is a quasitopos with small limits and small colimits.
\end{Theorem}

Theorem~\ref{thm:ConcQuasiCat} implies that the category of diffeological spaces has a number of convenient properties. For clarity and later reference, we will spell out some of them.

\begin{Proposition}
\label{prop:ConcSheafProperties}
The category $\Dflg$ of diffeological spaces has the following properties:
\begin{itemize}

\item $\Dflg$ is locally cartesian closed, i.e.~for every object $X$ in $\Dflg$ the overcategory $\Dflg\Comma X$ is cartesian closed.

\item Strong monomorphisms and strong epimorphisms are effective.

\item (Strong) monomorphisms and (strong) epimorphisms are stable under pullback.

\item $\Dflg$ is quasiadhesive, that is, the pushout of a strong monomorphism is a strong monomorphism and the pushout square is a pullback square.

\item The initial object is strict, i.e.~every morphism $X \to \emptyset$ is an isomorphism.

\item Coproducts are disjoint, i.e.~$X \to X \sqcup Y \leftarrow Y$ are monomorphisms and $X \times_{X \sqcup Y} Y \cong \emptyset$.

\item The functor of points $\Dflg \to \Set$, $X \to \Dflg(*,X)$ is faithful. It has a left and a right adjoint, so that it preserves limits and colimits.

\end{itemize}

\end{Proposition}

\begin{Terminology}
The strong epimorphisms in $\Dflg$ are called \textbf{subductions}, the strong monomorphisms \textbf{inductions}.
\end{Terminology}

\begin{Example}
\label{ex:Diffeologies}
Here are some of the most basic examples for diffeologies:
\begin{itemize}

\item[(a)] The \textbf{fine diffeology} or \textbf{discrete diffeology} on a set $S$ is the diffeology for which the plots are the locally constant maps.\footnote{In \cite[Example~(2), p.~5794]{BaezHoffnung:2011} it is stated incorrectly that the discrete diffeology is given by the constant maps.}

\item[(b)] The \textbf{coarse diffeology}, or \textbf{indiscrete diffeologoy}, or \textbf{trivial diffeology} on a set $S$ is given by $U \mapsto \Set\bigl(|U|, S\bigr)$, i.e.~all maps are plots.

\item[(c)] Every topological space $X$ is equipped with the \textbf{continuous diffeology} given by $U \mapsto  \Top(U,X)$, i.e.~the plots are the continuous maps.

\item[(d)] Every smooth finite-dimensional manifold $M$ is equipped with the \textbf{natural diffeology} given by $U \mapsto \Mfld(U,M)$, i.e.~the plots are the infinitely often differentiable maps.

\item[(e)] We will denote the exponential objects in $\Dflg$ by
\begin{equation*}
  \intHom(X,Y) \equiv Y^X
\end{equation*}
and call them the \textbf{diffeological mapping spaces}. The diffeology, which is given by the universal property
\begin{equation*}
  \Dflg\bigl(U, \intHom(X,Y) \bigr)
  \cong \Dflg(U \times X, Y)
  \,,
\end{equation*}
is called the \textbf{functional diffeology}.

\end{itemize}
\end{Example}

\begin{Remark}
\label{rmk:CoFreeDflg}
By Proposition~\ref{prop:ConcSheafProperties}, the forgetful functor $\Dflg \to \Set$, $X \mapsto |X|$ has a left and a right adjoint. The right adjoint equips a set $S$ with the coarse diffeology. The left adjoint equips it with the fine diffeology. In other words, the fine diffeology on a set is the free diffeology, the coarse diffeology is the cofree diffeology.
\end{Remark}

\begin{Remark}
\label{rmk:MlfdDifflgEmbed}
The map that sends a smooth manifold $M$ to its natural diffeology $U \mapsto \Mfld(U,M)$ defines a full, faithful, and injective functor $\Mfld \to \Dflg$.
\end{Remark}

\begin{Remark}
\label{rmk:D-topology}
Every diffeological space $X$ is naturally equipped with the finest topology such that all plots are continuous, which is called the $D$-topology. This topology is determined by the smooth curves only, so that many different diffeologies induce the same topology \cite[Thm.~3.7]{ChristensenSinnamonWu:2014}. Mapping a diffeology on $X$ to the induced topology is left adjoint to mapping a topology to the continuous diffeology \cite[Prop.~3.3]{ChristensenSinnamonWu:2014}. The topology induced by the discrete (trivial) diffeology is the discrete topology. In general, however, neither the unit nor the counit of the adjunction is an isomorphism.
\end{Remark}

The maps $p \in X(U) \subset \Set(|U|, |X|)$ of a diffeological space $X$ are called \textbf{plots}. If we spell out the defining conditions of a concrete sheaf, we obtain the traditional definition of diffeological spaces in terms of plots.

For every open cover $\{U_i \to U\}$ in $\Eucl$, $U$ is the coequalizer of $\coprod_{i,j} U_i \cap U_j \rightrightarrows \coprod U_i$. In other words, the site of euclidean spaces is subcanonical, so that every representable presheaf is a sheaf. Since every representable presheaf is concrete, it follows that all representable presheaves on $\Eucl$ are concrete sheaves. We conclude that the Yoneda embedding $Y: X \mapsto \Dflg(\Empty, X)$ factors as
\begin{equation*}
\begin{tikzcd}
& \Set^{\Eucl^\op}
\\
\Eucl \ar[ur, "Y"] \ar[r, "y"'] &
\Dflg \ar[u, "I"']
\end{tikzcd}    
\end{equation*}
through a functor $y$. Since $Y$ and $I$ are full and faithful, so is $y$. Since $I$ is full and faithful, the Yoneda lemma implies that the evaluation of the concrete sheaf $X \in \Dflg$ on $U \in \Eucl$ is given by
\begin{equation*}
\begin{split}
  X(U) 
  &\cong \Set^{\Eucl^\op}(YU, IX)
  \cong \Set^{\Eucl^\op}(IyU, IX)
  \\
  &\cong
  \Dflg(yU, X)
  \,.    
\end{split}
\end{equation*}
It follows that limits in $\Dflg$ are computed pointwise and that $I$ preserves limits. By the adjoint functor theorem, $I$ has a left adjoint,
\begin{equation}
\label{eq:PreshDflgAdjunction}
\begin{tikzcd}
K: \Set^{\Eucl^\op} \ar[r, shift right] &
\Dflg :I \ar[l, shift right]
\,,
\end{tikzcd}    
\end{equation}
which was computed and studied in \cite[Sec.~5.3]{BaezHoffnung:2011}. Explicitly, $K$ is given by a procedure called concretization followed by the Grothendieck plus construction.

The left adjoint $K$ is a retract, $KI \cong \id_{\Dflg}$, which implies that the colimit of a diagram in $X: \calI \to \Dflg$ can be computed as
\begin{equation*}
\begin{split}
  \Colim_{i \in \calI} X_i
  &\cong 
  \Colim_{i \in \calI} KIX_i
  \\
  &\cong 
  K \Colim_{i \in \calI} IX_i
  \,,
\end{split}
\end{equation*}
that is, by first computing the colimit in presheaves and then applying $K$. As a further consequence, it can be shown that $y$ is dense:

\begin{Proposition}[Prop.~51 in \cite{BaezHoffnung:2011}]
\label{prop:ColimRep}
Every $X \in \Dflg$ is the colimit of $y \Comma X \to \Eucl \to \Dflg$, which we will write as
\begin{equation}
\label{eq:ColimPlotCat}
  X \cong \Colim_{yU \to X} yU 
  \,.
\end{equation}
\end{Proposition}

\begin{Terminology}
The comma category $y \Comma X$ is called the \textbf{category of plots} of $X$.
\end{Terminology}

\begin{Remark}
\label{not:SubcatIdentify}
It is customary and convenient to identify notationally the domain of a plot $U \in \Eucl$ with the diffeological space $yU \in \Dflg$. In this paper, however, we deal with a number of subtleties of Kan extensions along $y$ where this identification would invite wrong proofs by notation (a trap the author has fallen into more than once). Therefore, we will always spell out the embedding $y$.
\end{Remark}

\subsection{Left Kan extension to diffeological spaces}
\label{sec:LeftKanExt}

\begin{Notation}
Let $\eFun{F}: \Eucl \to \Eucl$ be an endofunctor. The left Kan extension of $y\eFun{F}$ along the embedding $y: \Eucl \hookrightarrow \Dflg$ will be denoted by 
\begin{equation}
\label{eq:LanyyDef1}
  \Lanyy \eFun{F} := \Lan_y y \eFun{F}
  \,,
\end{equation}
which is an endofunctor of $\Dflg$.
\end{Notation}

The left Kan extension will be our device to extend the tangent structure of euclidean to diffeological spaces. $\Lanyy$ is functorial, which means that $\Lanyy$ preserves the vertical compostion of natural transformations $\hat{\alpha}: \hat{F} \to \hat{F}'$ and $\hat{\alpha}': \hat{F}' \to \hat{F}''$ between endofunctors $\hat{F}, \hat{F}', \hat{F}'' \in \End(\Eucl)$, that is,
\begin{equation}
\label{eq:KanVertComp}
  \Lanyy(\alpha' \circ \alpha)
  =   
  (\Lanyy\alpha') \circ (\Lanyy\alpha)
  \,.
\end{equation}

Since $\Eucl$ is small and $\Dflg$ is cocomplete, $\Lanyy \eFun{F}$ exists and is pointwise, that is, it can be computed by the colimit \cite[Thm.~X.5.3]{MacLane:Working}
\begin{equation}
\label{eq:LanColim}
\begin{split}
  (\Lanyy\eFun{F})(X)
  &\cong \Colim( 
  y \Comma X \longrightarrow 
  \Eucl \stackrel{y\eFun{F}}{\longrightarrow} 
  \Dflg)
  \\
  &= \Colim_{yU \to X} y \eFun{F} U 
  \,,
\end{split}
\end{equation}
for all $X \in \Dflg$. Let $\eFun{\alpha}: \eFun{F} \to \eFun{G}$ be a natural transformation of functors $\Eucl \to \Eucl$. Then $y\eFun{\alpha}: y\eFun{F} \to y\eFun{G}$ is a natural transformation of functors $\Eucl \to \Dflg$. The left Kan extension $\Lan_y y\eFun{F}$ is functorial in $y\eFun{F}$, so that we have a natural transformation
\begin{equation}
\label{eq:LanyyDef2}
  \Lanyy\eFun{\alpha} 
  := \Lan_y y\eFun{\alpha}: 
  \Lanyy \eFun{F}
  \longrightarrow
  \Lanyy\eFun{G} 
  \,.
\end{equation}
Together~\eqref{eq:LanyyDef1} and \eqref{eq:LanyyDef2} define a functor 
\begin{equation}
\label{eq:KanExtFunctor}
  \Lanyy: \End(\Eucl) \longrightarrow \End(\Dflg)
  \,,
\end{equation}
where $\End(\calC)$ denotes the category of endofunctors and natural transformations of the category $\calC$.

\begin{Proposition}
\label{prop:LanRestrict}
The diagram
\begin{equation*}
\begin{tikzcd}[column sep={large}]
\Eucl \ar[d, "y"'] \ar[r, "\eFun{F}"] & \Eucl \ar[d, "y"]
\\
\Dflg \ar[r, "\Lanyy\eFun{F}"'] & \Dflg
\end{tikzcd}    
\end{equation*}
commutes for all endofunctors $\eFun{F}$, that is
\begin{equation}
\label{eq:LanCommute}
  (\Lanyy\eFun{F})yU \cong y\eFun{F} U    
\end{equation}
for all $U \in \Eucl$. 
\end{Proposition}
\begin{proof}
Since $y$ is full and faithful, the statement follows from \cite[Cor.~3, Sec.~X.3]{MacLane:Working} or from \cite[Prop.~4.23]{Kelly:Enriched}.
\end{proof}


\begin{Corollary}
The functor~\eqref{eq:KanExtFunctor} is full and faithful.
\end{Corollary}

\subsection{Compatibility with products, coproducts, and subductions}

Since $y: \Eucl \to \Dflg$ is full and faithful, a smooth map $f: U \to V$ of euclidean spaces is a strong epimorphism if and only if $yf: yU \to yV$ is a strong epimorphism, which is the same thing as a subduction. For this reason we will call a strong epimorphism $f$ a subduction. This is the case if every point $v_0 \in V$ has an open neighborhood $V_0 \subset V$ such that there is a smooth map $g_0: V_0 \to U$ satisfying $f \circ g_0 = \id_{V_0}$. In short, $f$ is a subduction if it has local sections. In particular, every surjective submersion is a subduction.

\begin{Proposition}
\label{prop:KanPreserveSubduc}
Let $\eFun{\alpha}: \eFun{F} \to \eFun{G}$ be a natural transformation of endofunctors of $\Eucl$. If $\eFun{\alpha}_U: \eFun{F} U \to \eFun{G} U$ is a subduction for all $U \in \Eucl$, then $(\Lanyy\eFun{\alpha})_X: (\Lanyy\eFun{F})X \to (\Lanyy\eFun{G})X$ is a subduction for all $X \in \Dflg$.
\end{Proposition}
\begin{proof}
Since $\eFun{\alpha}_U$ is a subduction, so is $y\eFun{\alpha}_U$. Subductions in $\Dflg$ are the same as regular epimorphisms, so that $y\eFun{\alpha}_U$ is a regular epimorphism for all $U \in \Eucl$. The left Kan extension $\alpha_X = (\Lanyy\eFun{\alpha})_X$ is given by the colimit over the category of plots $y \Comma X$. Since colimits preserve regular epimorphisms, $\alpha_X$ is a regular epimorphism, that is, a subduction.
\end{proof}


\begin{Proposition}
\label{prop:FpreservesProducts}
If a functor $\eFun{F}: \Eucl \to \Eucl$ preserves finite products, then so does $\Lanyy\eFun{F}: \Dflg \to \Dflg$.
\end{Proposition}

Let $\eFun{F}: \calI \to \End(\Eucl)$, $i \mapsto \eFun{F}_i$ be a functor. Due to the universal properties of colimits and limits, we have for every $X \in \Dflg$ the natural morphism
\begin{equation*}
  \Colim_{yU \to X} \Lim_{i\in \calI} y\eFun{F}_i U
  \longrightarrow 
  \Lim_{i\in \calI} \Colim_{yU \to X} y\eFun{F}_i U 
  \,.
\end{equation*}
Assuming that the limit $\Lim_{i\in \calI} \eFun{F}_i$ exists in $\End(\Eucl)$, it can be written as the natural transformation
\begin{equation}
\label{eq:KanLimCommute}
  \Lanyy \Lim_i \eFun{F}_i \longrightarrow
  \Lim_i \Lanyy \eFun{F}_i
  \,,
\end{equation}
where we have used that $y$ preserves limits. This is not an isomorphism unless the colimit and the limit commute. 

Let $\eFun{G}: \calJ \to  \End(\Eucl)$ be another diagram, such that $\Lim_i \eFun{G}_i$ exists. Any natural transformation $\eFun{\alpha}_i: \eFun{F}_i \to \eFun{G}_i$ induces a commutative diagram
\begin{equation}
\label{eq:KanLimCommute2}
\begin{tikzcd}[column sep=large]
\Lanyy \Lim_i \eFun{F}_i 
\ar[r, "\Lanyy(\Lim_i \eFun{\alpha}_i)"]
\ar[d]
&
\Lanyy \Lim_i \eFun{G}_i
\ar[d]
\\
 \Lim_i \Lanyy \eFun{F}_i 
\ar[r, "\Lim_i \Lanyy\eFun{\alpha}_i"']
&
\Lim_i 
\Lanyy\eFun{G}_i
\end{tikzcd}
\end{equation}

\begin{Proposition}
\label{prop:LanyyProdFunc}
Let $\eFun{F}_1, \ldots, \eFun{F}_k \in \End(\Eucl)$ be a finite family of endofunctors. Then we have an isomorphism
\begin{equation*}
  \Lanyy(\eFun{F}_1 \times \ldots \times \eFun{F}_k)
  \cong
  \Lanyy\eFun{F}_1 \times \ldots \times \Lanyy\eFun{F}_k
  \,.
\end{equation*}
\end{Proposition}

It is well-known, that the left Kan extension of an arbitrary functor along the Yoneda embedding $Y:\Eucl \to \Set^{\Eucl^\op}$ preserves all colimits. This is not true for Kan extensions along $y: \Eucl \to \Dflg$. Already for the preservation of coproducts we have to make additional assumptions. Recall that a functor $\eFun{F}: \Eucl \to \calC$ is a \textbf{cosheaf} if $F^\op: \Eucl^\op \to \calC^\op$ is a sheaf.

\begin{Example}
The following functors on $\Eucl$ are cosheaves:
\begin{itemize}

\item[(a)] the tangent functor $\eFun{T}: \Eucl \to \Eucl$;

\item[(b)] fiber products of the tangent functor $\eFun{T}_k: \Eucl \to \Eucl$;

\item[(c)] if $\eFun{F}: \Eucl \to \calC$ and $\eFun{G}: \Eucl \to \Eucl$ are cosheaves, then so is their composition $\eFun{F}\eFun{G}$;

\item[(d)] the de Rham functor $\eFun{\Omega}: \Eucl \to \dgAlg^\op$, which maps $U$ to the differential graded algebra of differential forms and smooth maps to the pullback of forms;

\item[(d)] if $F:\Mfld^\op \to \calC$ is a sheaf on the big site of manifolds and open covers, then the restriction $\eFun{F}: \Eucl \hookrightarrow \Mfld \to \calC^\op$ is a cosheaf;

\end{itemize}
\end{Example}

\begin{Proposition}
\label{prop:KanCoproducts}
If $F: \Eucl \to \calC$ is a cosheaf, then its left Kan extension along $y: \Eucl \to \Dflg$ preserves coproducts.
\end{Proposition}

\begin{Corollary}
\label{cor:KanCoproducts2}
If $\eFun{F}: \Eucl \to \Eucl$ is a cosheaf, then $\Lanyy\eFun{F}$ preserves coproducts.    
\end{Corollary}

\begin{Proposition}
\label{prop:KanStrongEpis}
If $\eFun{F}: \Eucl \to \Eucl$ is a cosheaf, then $\Lanyy\eFun{F}$ preserves subductions.
\end{Proposition}

\begin{Proposition}
\label{prop:ColimSubduction}
Let $X:\calI \to \Dflg$ be a functor. If $\eFun{F}: \Eucl \to \Eucl$ is a cosheaf, then the natural morphism
\begin{equation*}
  \Colim (\Lanyy F)X \longrightarrow
  (\Lanyy F)\Colim X
\end{equation*}
is a subduction.
\end{Proposition}

\subsection{Compatibility with the composition of endofunctors}

Let $\eFun{F},\eFun{G}: \Eucl \to \Eucl$ be endofunctors. The left Kan extension of their product is given by the colimit
\begin{equation}
\label{eq:KanProdIso1}
\begin{split}
  \Lanyy (\eFun{F} \eFun{G})X
  &\cong
  \Colim_{yU \to X} y \eFun{F}\eFun{G} U
  \\
  &\cong
  \Colim_{yU \to X}\, (\Lanyy\eFun{F}) y\eFun{G} U
  \,,
\end{split}
\end{equation}
where we have used Proposition~\ref{prop:LanRestrict}. The product of the Kan extensions can be written as
\begin{equation}
\label{eq:KanProdIso2}
  (\Lanyy \eFun{F})(\Lanyy \eFun{G})X
  \cong
  (\Lanyy\eFun{F}) \Colim_{yU \to X} y\eFun{G} U
  \,.
\end{equation}
Due to the universal property of the colimit, we have a natural morphism
\begin{equation}
\label{eq:KanProdMorph1}
  \Colim_{yU \to X}\, (\Lanyy\eFun{F}) y\eFun{G} U
  \longrightarrow
  (\Lanyy\eFun{F}) \Colim_{yU \to X} y\eFun{G} U
  \,.
\end{equation}
Composing this morphism with the isomorphism~\eqref{eq:KanProdIso1} on its domain and with isomorphism~\eqref{eq:KanProdIso2} on the codomain, we obtain a natural morphism
\begin{equation}
\label{eq:yKanMonoid}
  \mu_{\eFun{F},\eFun{G}}:
  \Lanyy(\eFun{F}\eFun{G} )X
  \longrightarrow 
  (\Lanyy\eFun{F}) (\Lanyy\eFun{G})X
  \,.
\end{equation}
This is an isomorphism if and only if~\eqref{eq:KanProdMorph1} is, which is generally not the case. 

\begin{Proposition}
\label{prop:KanProdSubduction}
If $\eFun{F}$ is a cosheaf, then~\eqref{eq:yKanMonoid} is a subduction for all $X \in \Dflg$.
\end{Proposition}
\begin{proof}
Since $\eFun{F}$ is a cosheaf, it follows from Proposition~\ref{prop:ColimSubduction} that \eqref{eq:KanProdMorph1}
is a subduction, which implies that~\eqref{eq:yKanMonoid} is a subduction.
\end{proof}

The morphism~\eqref{eq:yKanMonoid} is associative in the followoing sense.

\begin{Proposition}
\label{prop:KanAssoc}
Let $\eFun{F},\eFun{G}, \eFun{H}: \Eucl \to \Eucl$ be endofunctors. Then the diagram of natural transformations
\begin{equation}
\label{eq:KanAssoc0}
\begin{tikzcd}[column sep=large]
\Lanyy(\eFun{F} \eFun{G} \eFun{H}) 
\ar[r, "\mu_{\eFun{F},\eFun{G}\eFun{H}}"] 
\ar[d, "\mu_{\eFun{F}\eFun{G},\eFun{H}}"']
&
(\Lanyy\eFun{F}) (\Lanyy\eFun{G} \eFun{H}) 
\ar[d, "(\Lanyy\eFun{F})\mu_{\eFun{G},\eFun{H}}"]
\\
\Lanyy(\eFun{F} \eFun{G}) (\Lanyy\eFun{H}) 
\ar[r, "\mu_{\eFun{F},\eFun{G}}(\Lanyy\eFun{H})"']
&
(\Lanyy\eFun{F}) (\Lanyy\eFun{G})(\Lanyy\eFun{H}) 
\end{tikzcd}    
\end{equation}
is commutative.
\end{Proposition}

\begin{Proposition}
\label{prop:KanHorComp}
Let $\eFun{\alpha}: \eFun{F} \to \eFun{F}'$ and $\eFun{\beta}: \eFun{G} \to \eFun{G}'$ be natural transformations of endofunctors of $\Eucl$. Then the diagram
\begin{equation*}
\begin{tikzcd}[column sep=large]
\Lanyy(\eFun{F} \eFun{G})
\ar[r, "\Lanyy(\eFun{\alpha}\eFun{\beta})"] 
\ar[d, "\mu_{\eFun{F},\eFun{G}}"'] 
&
\Lanyy(\eFun{F}'\eFun{G}') 
\ar[d, "\mu_{\eFun{F}',\eFun{G}'}"]
\\
(\Lanyy\eFun{F})(\Lanyy\eFun{G}) 
\ar[r, "(\Lanyy\eFun{\alpha})(\Lanyy\eFun{\beta})"']
&
(\Lanyy\eFun{F}') (\Lanyy\eFun{G}') 
\end{tikzcd}    
\end{equation*}
is commutative.    
\end{Proposition}

\subsection{Compatibility with the \textit{D}-topology}

\begin{Definition}[Sec.~2.8 in \cite{Iglesias-Zemmour:Diffeology}]
The \textbf{$D$-topology} on a diffeological space is the finest topology (on the underlying set) such that every plot is continuous.
\end{Definition}

Explicitly, a subset $Y \subset X$ is open in the $D$-topolgy if and only if for every plot $p: yU \to X$, the preimage $p^{-1}(Y) \subset U$ is open. Every morphism of diffeological spaces is continuous with respect to the $D$-topologies. In the following ``open'' and ''continuous'' are always meant with respect to the $D$-topology. An open subset $S \subset X$ of a diffeological space is naturally equipped with the subspace diffeology, so that the inclusion $i:S \to X$ is an open induction.

Let $\{S_i \subset X\}_{i \in I}$ be an open cover of a diffeological space $X$. Then the diagram
\begin{equation}
\label{eq:LocFun1a}
\begin{tikzcd}
\coprod_{i,j} S_{ij}
\ar[r, shift left] \ar[r, shift right] &
\coprod_i S_i \ar[r] & 
X
\end{tikzcd}
\,,
\end{equation}
where $S_{ij} = S_i \times_X S_j$, is a coequalizer. A functor $F: \Dflg \to \Dflg$ is a \textbf{cosheaf} if $F$ preserves the coequalizer~\eqref{eq:LocFun1a}, that is, if
\begin{equation}
\label{eq:LocFun1b}
\begin{tikzcd}
\coprod_{i,j} FS_{ij}
\ar[r, shift left] \ar[r, shift right] &
\coprod_i FS_i \ar[r] & 
FX
\end{tikzcd}
\end{equation}
is a coequalizer for every open cover. A special kind of cosheaf is given by a natural bundle that satisfies the following locality condition.

\begin{Definition}
\label{def:LocalBundle}
A bundle $\pi: F \to \Id$ of endofunctors of diffeological spaces will be called \textbf{local} if for every open induction $i: S \to X$ the commutative diagram
\begin{equation}
\label{eq:Dtop2}
\begin{tikzcd}
FS 
\ar[d, "\pi_S"']
\ar[r, "Fi"]
&
FX \ar[d, "\pi_X"]
\\
S \ar[r, "i"'] &
X
\end{tikzcd}
\end{equation}
is a pullback.
\end{Definition}

\begin{Proposition}
\label{prop:LocBundCosheaf}
If a bundle $\pi: F \to \Id$ of endofunctors of diffeological spaces is local, then $F$ is a cosheaf.
\end{Proposition}

\begin{Proposition}
\label{prop:LocBudInduction}
If a bundle $\pi: F \to \Id$ of endofunctors of diffeological spaces is local, then $F$ preserves open inductions.
\end{Proposition}

\begin{Proposition}
\label{prop:LocBundProd}
Let $\pi: F \to \Id$ and $\rho: G \to \Id$ be bundles of endofunctors of diffeological spaces. If both bundles are local, then the fiber product $F \times_\Id G \to \Id$ and the composition $FG \to \Id$ is local. 
\end{Proposition}

\begin{Proposition}
\label{prop:LocFuncCover}
Let $\pi: F \to \Id$ and $\rho: G \to \Id$ be bundles of endofunctors of diffeological spaces and $\alpha: F \to G$ a morphism of bundles, $\rho \circ \alpha = \pi$. If both bundles are local, then the following are equivalent:
\begin{itemize}

\item[(i)] $\alpha_X: FX \to GX$ is an induction (subduction, epimorphism, isomorphism).

\item[(ii)] There is an open cover $\{S_i \to X \}$, such that $\alpha_{S_i}: FS_i \to GS_i$ is an induction (subduction, epimorphism, isomorphism) for all $i$.

\end{itemize}
\end{Proposition}

Definition~\ref{def:LocalBundle} applies also to bundles of endomorphisms of euclidean spaces. Explicitly, a bundle $\eFun{\pi}: \eFun{F} \to \Id$ is local if $\eFun{\pi}_U^{-1}(V) \cong \eFun{F}V$ for every open subset $V \subset U$. The next proposition shows that this notion locality is preserved by the left Kan extension to diffeological spaces.

\begin{Proposition}
\label{prop:KanLocal}
If a bundle $\eFun{\pi}: \eFun{F} \to 1$ of endofunctors of euclidean spaces is local, then so is its left Kan extension $\Lanyy\eFun{\pi}: \Lanyy\eFun{F} \to \Id$.
\end{Proposition}

\section{Elastic diffeological spaces}
\label{sec:Elastic}

The main result for the category of elastic spaces is that the left Kan extension of the tangent structure on euclidean spaces defines a tangent structure with scalar $\bbR$-mul\-ti\-pli\-ca\-tion (Theorem~\ref{thm:ElasticTanCat}). The proof of this statement is quite long and will be given in \cite{Blohmann:Elastic}. It uses a variety of techniques, some categorical, some geometric, and all of the results about the left Kan extension given in the preceeding Section~\ref{sec:LeftKanExt}.

\subsection{The tangent structure of elastic spaces}

The tangent structure of euclidean spaces consists of the tangent functor $\hat{T}: \Eucl \to \Eucl$ together with the natural transformations of the bundle projection $\hat{\pi}: \hat{T} \to \Id$, the zero section $\hat{0}: \Id \to \hat{T}$, the addition $\hat{+}: \hat{T}_2 \to \hat{T}$, the symmetric structure $\hat{\tau}: \hat{T}^2 \to \hat{T}^2$, and the vertical lift $\hat{\lambda}: \hat{T} \to \hat{T}^2$, which we from now one decorate with hats in order to distinguish them from their Kan extensions to diffeological spaces. The structure was spelled out explicitly in Section~\ref{sec:TangentEuclidean}.

We would like to extend this structure to diffeological spaces by applying the Kan extension functor $\Lanyy = \Lan_y y\Empty$. The Kan extension of the tangent functor will be denoted by
\begin{equation}
\label{eq:TDflgDef}
  T := \Lanyy \hat{T} : \Dflg \rightarrow \Dflg
  \,.
\end{equation}

\begin{Definition}
\label{def:elastic}
A diffeological space $X$ is called \textbf{elastic} if the following axioms hold:
\begin{itemize}
    
\item[(E1)] The natural morphisms
\begin{equation*}
  \theta_{k,X}: (\Lanyy \hat{T}_k) X 
  \longrightarrow T_k X
\end{equation*}
are isomorphisms for all $k > 1$.

\item[(E2)] There is a natural morphism $\tau_X: T^2 X \to T^2 X$, such that the diagram
\begin{equation*}
\begin{tikzcd}
(\Lanyy\hat{T}^2) X 
\ar[r, "(\Lanyy\hat{\tau})_X"] 
\ar[d, "\theta^2_X"']
&
(\Lanyy\hat{T}^2) X 
\ar[d, "\theta^2_X"]
\\
T^2 X \ar[r, "\tau_X"'] 
&
T^2 X
\end{tikzcd}    
\end{equation*}
commutes.

\item[(E3)] The natural morphism
\begin{equation*}
  \lambda_X: TX \xrightarrow{~(\Lanyy\hat{\lambda})_X~} 
  (\Lanyy \hat{T}^2)X
  \xrightarrow{~\theta^2_X~}
  T^2 X
\end{equation*}
is an induction.

\item[(E4)] The natural morphisms
\begin{equation*}
  \nu_{k,X} : T T_k X \longrightarrow
  T^2 X  \times_{TX}^{T\pi_X,T\pi_X} 
  \ldots \times_{TX}^{T\pi_X,T\pi_X}
  T^2 X
\end{equation*}
are injective for all $k > 1$.

\item[(E5)] For every finite set of positive integers $k_1, \ldots, k_n$ the diffeological space $X' := T_{k_1} \cdots T_{k_n} X$ satisfies axioms (E1) through (E4).

\end{itemize}
The full subcategory of elastic diffeological spaces will be denoted by $\Elst \subset \Dflg$.
\end{Definition}

\begin{Theorem}
\label{thm:ElasticTanCat}
The category of elastic diffeological spaces has a tangent structure given by the Kan extended tangent functor $T = \Lanyy\hat{T}$, bundle projection $\pi_X := (\Lanyy\hat{\pi})_X$, and zero section $0_X := (\Lanyy\hat{0})_X$, the addition
\begin{equation*}
  +_X := \theta_{2,X}^{-1} \circ (\Lanyy\hat{+})_X \circ \theta_{2,X}
  \,,
\end{equation*}
the symmetric structure $\tau_X$ of Axiom~(E2) of Definition~\ref{def:elastic}, and the vertical lift $\lambda_X$ of Axiom~(E3). Moreover, the Kan extension of the scalar $\bbR$-multiplication is a scalar $\bbR$-multiplication in the sense of Definition~\ref{def:RScalar}.
\end{Theorem}

\subsection{Alternative axioms}

The axioms for abstract tangent structures are to hold for the entire category. This is why the Axiom~(E5) has to be included in the Definition~\ref{def:elastic} of elastic spaces, so that the application of $T_k$ does not lead out of the subcategory of elastic spaces. However, Rosick\'y's axioms still make sense pointwise for a single object $X$ that lies in an ambient category where $T_k$ is defined. This is our situation, which suggests the following more general concept.

\begin{Definition}
\label{def:weakelastic}
A diffeological space $X$ is called \textbf{weakly elastic} if the Axioms~(E1)-(E4) of Definition~\ref{def:elastic} hold.
\end{Definition}

On a weakly elastic space we still have most of the structure of differential calculus, like a Lie algebra of vector fields. We only have to observe that $TX$ and its fiber products may no longer share the same good properties.

\begin{Remark}
Theorem~\ref{thm:ElasticTanCat} still holds if Axiom~(E1) of Definition~\ref{def:elastic} is replaced with the following weaker version:
\begin{itemize}
\item[(E1')] The natural morphism
\begin{equation*}
  \theta_{k,X}: (\Lanyy \hat{T}_k) X 
  \longrightarrow T_k X
\end{equation*}
is an isomorphism for $k=2$ and an epimorphism for all $k > 2$.
\end{itemize}
As explained in the introduction, we need the stronger Axiom~(E1) so that we obtain a Cartan calculus.     
\end{Remark}

In earlier versions of the definition of elastic space, we have used the following axiom:
\begin{itemize}
\item[(E0)] The natural morphisms
\begin{equation}
\label{eq:GrenobleAxiom}
  \Lanyy(\hat{T}_k \hat{T}) X \longrightarrow T_k T X
\end{equation}
are isomorphisms for all $k \geq 1$.
\end{itemize}

\begin{Proposition}
The Axiom~(E0) implies the Axioms~(E1), (E2), and (E4) of Definition~\ref{def:elastic}.
\end{Proposition}

While this proposition shows that Axiom~(E0) is logically stronger than (E1), (E2), and (E4), we currently do not know an example of an elastic space that does not satisfy the stronger Axiom~(E0).

\subsection{Stability properties of elastic spaces}

\begin{Proposition}
\label{prop:OpenElastic}
Let $X$ be a diffeological space. The following are equivalent:
\begin{itemize}

\item[(i)] $X$ is elastic.

\item[(ii)] $X$ has an open cover $\{S_i \to X \}$ by elastic spaces $S_i$.

\end{itemize}
\end{Proposition}

\begin{Corollary}
\label{cor:MfldElastic}
Let $X$ be a diffeological manifold modelled on the diffeological vector space $A$. Then $X$ is elastic if and only if $A$ is elastic.
\end{Corollary}

\begin{Corollary}
\label{cor:CoproductsElastic}
Let $\{X_i\}_{i \in I}$ be a small family of diffeological spaces. The following are equivalent:
\begin{itemize}

\item[(i)] The coproduct $\sqcup_i X_i$ is elastic.

\item[(ii)] Every $X_i$ is elastic.

\end{itemize}
\end{Corollary}

\begin{Proposition}
\label{prop:ProductElastic}
Finite products of elastic spaces are elastic.
\end{Proposition}

\begin{Proposition}
\label{prop:RetractsElastic}
Retracts of elastic spaces are elastic.
\end{Proposition}

\section{Examples}
\label{sec:Examples}

The conditions for a diffeological space to be elastic are quite strong. Limits, colimits, subspaces, and mapping spaces of elastic spaces are generally no longer elastic. If by considering elastic spaces we have given away many of the conventient properties of the category of diffeological spaces, the question arises whether there are enough interesting examples and applications for the concept to be useful.

We have already stated that finite products and small coproducts of elastic spaces are elastic. We have also seen that elastic spaces are stable under retracts, which allows for the construction of elastic spaces with rather benign singular behaviour. Basic examples are manifolds with corners \cite{Joyce:2012} and cusps.

The original example that motivated the concept of diffeology by Souriau is that of diffeological groups. Our main result states that a diffeolgical group is elastic if and only if the vertical lift $\lambda_G: TG \to T^2 G$ is an induction (Theorem~\ref{thm:ElasticGroups}). This mild condition is needed to avoid that the bracket of vector fields takes values in a tangent bundle with a weaker diffeology, as is the case for vector fields on groups on a $C^k$-manifold. This shows that almost all diffeological groups that come to mind are elastic. Even a seemingly pathological group like the quotient $\bbR/\bbQ$ is elastic. 

Every diffeological vector space has an underlying abelian group, so that we can apply the characterization of diffeological groups. It follows that a diffeological vector space $A$ is elastic if the natural map $T_0 A \to T_0 T_0 A$ that maps $v_a$ to the tangent vector represented by the path $t \mapsto t v_a$ is an induction. A diffeological manifold modelled on $A$ is elastic if and only if $A$ is elastic. A stronger condition which ensures elasticity is that $A \to T_0 A$ is an isomorphism, which yields a trivialization $TA \cong A \times A$ of the tangent bundle. We call such vector spaces tangent stable. All fine diffeological vector spaces have this property.

If $A$ is tangent stable then the diffeological mapping space $\intHom(X,A)$ is elastic for all $X$. An important example is the algebra $C^\infty(X) \cong \intHom(X,\bbR)$ of smooth functions on $X$. The diffeological space of sections $\Gamma(M,F)$ of a smooth fiber bundle $\rho: F \to M$ is elastic (Theorem~\ref{thm:FieldsElastic}). As expected, its tangent space is given by the sections of the vertical tangent bundle $\ker T\rho \to M$. As a corollary, the diffeological space $\intHom(M,N)$ of smooth maps of manifolds is elastic, the tangent space being given by
\begin{equation*}
  T\intHom(M,N) \cong \intHom(M, TN)
  \,.
\end{equation*}
More examples can be constructed by forming retracts, which allows for mildly singular situations.

\subsection{Manifolds with corners}

Consider the set $[0,\infty) \subset \bbR$ equipped with the subspace diffeology. Every smooth map $p \in U \to \bbR$ with image $p(U) \subset [0,\infty)$ has vanishing derivatives to all orders at every point $u_0$ with $p(u_0) =  0$. It follows that
\begin{equation*}
  T_0 [0,\infty) = 0 
  \,.
\end{equation*}
The tangent space at an interior point $x \in (0,\infty)$ is given by $T_x [0,\infty) = \bbR$. 

We want to show that $[0,\infty)$ is elastic. For this, we consider the maps
\begin{align*}
  \pi: \bbR &\longrightarrow [0,\infty)
  \\
  x &\longmapsto x^2 
  \\
\intertext{and}
  \sigma: [0, \infty) &\longrightarrow \bbR
  \\
  x &\longmapsto \sqrt{x} 
  \,,
\end{align*}
which satisfy $\pi \circ \sigma = \id_{[0,\infty)}$ as maps of sets. We have to show that $\sigma$ is smooth. Let $p: U \mapsto [0, \infty)$ be a plot. Since $\sigma$ is smooth on the interior $(0, \infty)$, $\sigma \circ p$ is smooth at a all points in $U$ that are mapped to the interior $(0, \infty)$. Assume that $p(u_0) = 0$. Since all derivatives of a plot $p$ vanish at $u_0$, $p$ vanishes to all orders at $u_0$, i.e.~
\begin{equation*}
  \frac{p(u)}{\| u - u_0\|^k} \xrightarrow{u \to u_0}
  0 \,,
\end{equation*}
for all $k \geq 0$. This implies that $\sqrt{p(u)} = (\sigma \circ p)(u)$ vanishes to all orders at $u_0$ as well, so that $\sigma \circ p$ is differentiable at $u_0$. We conclude that $\sigma$ is smooth, so that $[0,\infty)$ is a smooth retract of $\bbR$.

By Proposition~\ref{prop:RetractsElastic}, we conclude that $[0,\infty)$ is elastic and by Proposition~\ref{prop:ProductElastic} that any finite product
\begin{equation*}
  \bbR^n_k := [0,\infty)^k \times \bbR^{n-k}   
\end{equation*}
is elastic. Since the diffeological tangent functor commutes with products, the tangent spaces are given by
\begin{equation}
  T_{(x_1, \ldots, x_n)} \bbR^n_k
  = T_{x_1} [0,\infty) \times \ldots
    T_{x_k} [0,\infty) \times \bbR^{n-k}
  \,.
\end{equation}
Finally, it follows from Proposition~\ref{prop:OpenElastic} that every diffeological space modeled locally on $\bbR^n_k$ is elastic. Such spaces are called \textbf{manifolds with corners} \cite{Joyce:2012}. We conclude:

\begin{Proposition}
\label{prop:MfldCornersElastic}
Manifolds with corners are elastic.
\end{Proposition}

\subsection{Manifolds with cusps}

Consider the following subset of $\bbR^2$,
\begin{equation*}
  X := \bigl\{(x,y)~|~ x \geq 0 ~\wedge~ |y| \leq x \bigr\}
  \,,
\end{equation*}
with the subspace diffeology. We can squeeze or stretch the corner at $(0,0)$ by multiplying the $y$ coordinate by a smooth function $f \in C^\infty\bigl( [0,\infty) \bigr)$ (see Figure~\ref{fig:Squeezing}),
\begin{equation*}
\begin{aligned}
  \phi: X &\longrightarrow X_f \\
  (x,y) &\longmapsto \bigl(x, f(x)\,y\bigr)  \,,
\end{aligned}    
\end{equation*}
where
\begin{equation*}
  X_f := \bigl\{(x,y)~|~ x \geq 0 ~\wedge~ |y| \leq f(x) \bigr\}
  \,.
\end{equation*}

\begin{figure}
\centering

\begin{tikzpicture}

\node (A) at (0,0)
{
    \begin{tikzpicture}[bolli,domain=-1:1]
        \filldraw 
            plot[parametric] ({abs(\x)}, {\x});
        \draw (0,0) node [point] {};
    \end{tikzpicture}
};

\node (B) [right=of A]
{
    \begin{tikzpicture}[bolli]
        \path [->] (0,0) edge node [above] {$\cong$} (1,0);
    \end{tikzpicture}
};

\node (C) [right=of B]
{
    \begin{tikzpicture}[bolli,domain=-1:1]
        \filldraw 
            plot[parametric] ({(\x)^2}, {(\x)^5});
        \draw (0,0) node [point] {};
    \end{tikzpicture}
};

\end{tikzpicture}
\caption{Squeezing a corner by multiplying the $y$-coordinate with the function $f(x) = x^\frac{3}{2}$. The boundary of the resulting diffeological subspace of $\bbR^2$ is the curve $x^2 = y^5$ with a cusp at $(0,0)$.}
\label{fig:Squeezing}
\end{figure}
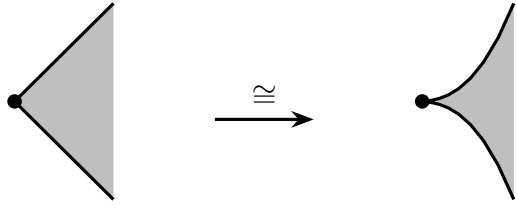

Assume that $f(x) > 0$ for $x > 0$. Then $\phi$ has an inverse map given by $\phi^{-1}(0,0) = (0,0)$ and 
\begin{equation*}
  \phi^{-1}(x,y) = 
  \Bigl(x, \frac{y}{f(x)} \Bigr)
\end{equation*}
otherwise. Let us equip $X_f$ with the pullback diffeology of $\phi^{-1}$. Since $\phi^{-1}$ is surjective, $\phi$ is an isomorphism of diffeological spaces. $X$ is $[0,\infty) \times [0,\infty)$ rotated by minus 45 degrees, so it is elastic by Proposition~\ref{prop:MfldCornersElastic}. Since $\phi$ is an isomorphism $X_f$ is elastic, too.

We cannot describe here in general, what a manifold with cusps or similar defects is, since any squeezing operation that is a smooth retract will produce a new model for elastic subspaces of $\bbR^n$. In Figure~\ref{fig:CornersCusps} on page~\pageref{fig:CornersCusps} we have given a few examples of elastic subspaces of $\bbR^2$ that can be obtained in this way.

\subsection{Diffeological groups}

A diffeological group is a group object $(G,m,e)$ in $\Dflg$. As for ordinary groups, the inverse $i: G \to G$, $g \mapsto g^{-1}$ is unique, so it is a property rather than a structure. By applying an endofunctor $F: \Dflg \to \Dflg$ that preserves products, we obtain a diffeological group $(FG, Fm, Fe)$. Examples for $F$ are $T$ and $T_k$ (see Proposition~\ref{prop:FpreservesProducts}). $(TG, Tm, Te)$ is called the tangent group.

From the multiplication of the tangent group we obtain the left $G$-translation of $TG$ given by
\begin{equation*}
  L: G \times TG 
  \xrightarrow{~0_G \times \id_{TG}~} 
  TG \times TG \xrightarrow{~Tm~} TG
  \,.
\end{equation*}
If we restrict the second argument to the tangent fiber at $e$, we obtain the map
\begin{equation}
\label{eq:Group01}
\begin{aligned}
  \phi_G:
  G \times T_e G 
  &\longrightarrow TG
  \\
  (g, v_e) &\longmapsto 
  L_g v_e
  \,,
\end{aligned}    
\end{equation}
where $T_e G = \{e\} \times_G TG$ is the tangent fiber at the group identity. It has an inverse given by
\begin{equation*}
  \phi_G^{-1}:
  TG \xrightarrow{~(i \circ \pi_G, \id_{TG})}
  G \times TG 
  \xrightarrow{~L~} T_e G
  \,,
\end{equation*}
which maps $v_g \mapsto L_{g^{-1}} v_g$. We conclude that~\eqref{eq:Group01} is an isomorphism. In analogy to Lie groups, we will denote
\begin{equation*}
  \frakg := T_e G \,,    
\end{equation*}
so that we obtain the trivialization of the tangent bundle,
\begin{equation*}
  TG \cong G \times \frakg
  \,.
\end{equation*}
$TG$ is itself a diffeological group with neutral element $0_e$. Using the isomorphism
\begin{equation*}
\begin{split}
  T_{(x,y)} (X,Y) 
  &\cong \{(x,y)\} \times_{X\times Y} T(X \times Y)
  \\
  &\cong (\{x\} \times_X TX) \times (\{y\} \times_Y TY)
  \\
  &\cong T_x X \times T_y Y
  \,,
\end{split}    
\end{equation*}
we obtain the trivialization
\begin{equation*}
\begin{split}
  T^2 G 
  &\cong
  TG \times T_{(0_e)} TG
  \\
  &\cong G \times \frakg \times T_{(e,0)}(G \times \frakg)
  \\
  &\cong G \times \frakg \times T_e G \times T_0 \frakg
  \\
  &\cong G \times \frakg \times \frakg \times T_0 \frakg
  \,.
\end{split}
\end{equation*}
In this trivialization the vertical lift $\lambda_G: TG \to T^2 G$ is given by
\begin{equation*}
\begin{aligned}
  G \times \frakg &\longrightarrow
  G \times \frakg \times \frakg \times T_0 \frakg
  \\
  (g,a) &\longmapsto
  \bigl(g,0,0,\lambda^\perp_{\frakg}(a) \bigr)
\end{aligned}
\end{equation*}
where
\begin{equation}
\label{eq:VecVertLift}
  \lambda^\perp_\frakg: \frakg \longrightarrow T_0 \frakg
  \,,
\end{equation}
maps $a \in \frakg$ to the tangent vector represented by the path $t \mapsto ta$.

\begin{Theorem}
\label{thm:ElasticGroups} 
A diffeological group $G$ is elastic if and only if $\lambda^\perp_\frakg: \frakg \to T_0 \frakg$ is an induction.
\end{Theorem}

It follows from Theorem~\ref{thm:ElasticGroups} and Theorem~\ref{thm:ElasticTanCat} that $\frakg$ and $T_0 \frakg$ are diffeological vector spaces. For an elastic group $\lambda^\perp_\frakg$ is an isomorphism and $\frakg$ is equipped with the Lie bracket of invariant vector fields.

\begin{Example}
Every Lie group $G$ is elastic when $G$ is a smooth manifold. However, when equipp $G$ with the $C^k$-diffeology, then $\frakg = \bbR^n_{C^k}$ and $T_0 \frakg = \bbR^n_{C^{k-1}}$, so that $\lambda^\perp_\frakg$ is no longer an induction (cf.~Example~\ref{ex:CkDiffeologies}).
\end{Example}

\begin{Example}
The diffeomorphism group of a smooth manifold is elastic.
\end{Example}

\subsection{Diffeological vector spaces}

A diffeological vector space is an $\bbR$-vector space object in $\Dflg$. Explicitly, this means that the addition and scalar multiplication are morphisms of diffeological spaces. Diffeological vector spaces are a rich and subtle structure \cite{ChristensenWu:2019}.

\begin{Proposition}
\label{prop:ElasticVect}
A diffeological vector space $A$ is elastic if and only if the natural linear map $T_0 A \to T_0 T_0 A$ is an induction.
\end{Proposition}
\begin{proof}
This is Theorem~\ref{thm:ElasticGroups} for the additive diffeological group $A$.
\end{proof}

In many cases, elastic diffeological vector spaces satisfy the stronger condition $A \cong T_0 A$, which is equivalent to 
\begin{equation*}
  TA \cong A \times A
  \,.
\end{equation*}
We will call diffeological vector spaces with this property \textbf{tangent stable}. Tangent stable vector spaces are elastic.

\begin{Proposition}
\label{prop:LiftA}
All fine diffeological vector spaces are tangent stable (hence elastic).
\end{Proposition}

\subsection{Diffeological manifolds}

We recall that a \textbf{diffeological manifold} modelled on the diffeological vector space $A$ is a diffeological space $X$ such that every point of $X$ has an open neighborhood that is isomorphic as diffeological space to an open subset of $A$. 

\begin{Proposition}
A diffeological manifold is elastic if and only if it is modelled on an elastic diffeological vector space.
\end{Proposition}
\begin{proof}
This follows from Proposition~\ref{prop:OpenElastic}.
\end{proof}

\subsection{Mapping spaces}

If $A$ is diffeological vector space, then the mapping space $\intHom(X,A)$ is a diffeological vector space with pointwise addition and scalar multiplication.

\begin{Proposition}
\label{prop:THomXAstable}
Let $X$ be a diffeological space and $A$ a diffeological vector space. If $A$ is tangent stable, then so is $\intHom(X,A)$.
\end{Proposition}

\begin{Corollary}
If $A$ is a tangent stable diffeological vector space then we have a natural isomorphisms
\begin{equation*}
  T\intHom(X,A) \cong
  \intHom(X,A) \times \intHom(X,A) \cong
  \intHom(X,TA)
\end{equation*}
for all $X \in \Dflg$.
\end{Corollary}

\begin{Corollary}
The diffeological space 
\begin{equation*}
  C^\infty(X) := \intHom(X,\bbR)   
\end{equation*}
of smooth $\bbR$-valued functions on a diffeological space $X$ is elastic.
\end{Corollary}

\begin{Theorem}
\label{thm:FieldsElastic}
Let $\rho: F \to M$ be a fiber bundle of smooth manifolds. Then the diffeological space of sections $\Gamma(M, F)$ is elastic with tangent space 
\begin{equation*}
  T\Gamma(M,F) \cong \Gamma(M,VF)
  \,,
\end{equation*}
the space of sections of the vertical tangent bundle $VF := \ker T\rho \to M$.
\end{Theorem}

\begin{Corollary}
\label{cor:SmoothMapElastic}
The diffeological mapping space $\intHom(M,N)$ of smooth manifolds $M$ and $N$ is elastic with tangent space 
\begin{equation*}
  T\intHom(M,N) \cong \intHom(M,TN)
  \,.
\end{equation*}
\end{Corollary}

\begin{Corollary}
The diffeological vector space of sections $\Gamma(M,A)$ of a smooth vector bundle $A \to M$ is tangent stable.
\end{Corollary}

\bibliographystyle{alpha}
\bibliography{ClassicalFields}

\begin{thebibliography}{CSW14}

\bibitem[BH11]{BaezHoffnung:2011}
John~C. Baez and Alexander~E. Hoffnung.
\newblock Convenient categories of smooth spaces.
\newblock {\em Trans. Amer. Math. Soc.}, 363(11):5789--5825, 2011.

\bibitem[Blo]{Blohmann:Elastic}
Christian Blohmann.
\newblock Tangent structure and {C}artan calculus on elastic diffeological
  spaces.
\newblock In preparation.

\bibitem[CC14]{CockettCruttwell:2014}
J.~R.~B. Cockett and G.~S.~H. Cruttwell.
\newblock Differential structure, tangent structure, and {SDG}.
\newblock {\em Appl. Categ. Structures}, 22(2):331--417, 2014.

\bibitem[CC15]{CockettCruttwell:2015}
J.~R.~B. Cockett and G.~S.~H. Cruttwell.
\newblock The {J}acobi identity for tangent categories.
\newblock {\em Cah. Topol. G\'{e}om. Diff\'{e}r. Cat\'{e}g.}, 56(4):301--316,
  2015.

\bibitem[CSW14]{ChristensenSinnamonWu:2014}
J.~Daniel Christensen, Gordon Sinnamon, and Enxin Wu.
\newblock The {$D$}-topology for diffeological spaces.
\newblock {\em Pacific J. Math.}, 272(1):87--110, 2014.

\bibitem[CW16]{ChristensenWu:2016}
J.~Daniel Christensen and Enxin Wu.
\newblock Tangent spaces and tangent bundles for diffeological spaces.
\newblock {\em Cah. Topol. G\'{e}om. Diff\'{e}r. Cat\'{e}g.}, 57(1):3--50,
  2016.

\bibitem[CW19]{ChristensenWu:2019}
J.~Daniel Christensen and Enxin Wu.
\newblock Diffeological vector spaces.
\newblock 2019.

\bibitem[DF99]{DeligneFreed:1999}
Pierre Deligne and Daniel~S. Freed.
\newblock Classical field theory.
\newblock In {\em Quantum fields and strings: a course for mathematicians,
  {V}ol. 1, 2 ({P}rinceton, {NJ}, 1996/1997)}, pages 137--225. Amer. Math.
  Soc., Providence, RI, 1999.

\bibitem[IZ13]{Iglesias-Zemmour:Diffeology}
Patrick Iglesias-Zemmour.
\newblock {\em Diffeology}, volume 185 of {\em Mathematical Surveys and
  Monographs}.
\newblock American Mathematical Society, Providence, RI, 2013.

\bibitem[Joy12]{Joyce:2012}
Dominic Joyce.
\newblock On manifolds with corners.
\newblock In {\em Advances in geometric analysis}, volume~21 of {\em Adv. Lect.
  Math. (ALM)}, pages 225--258. Int. Press, Somerville, MA, 2012.

\bibitem[Kel05]{Kelly:Enriched}
G.~M. Kelly.
\newblock Basic concepts of enriched category theory.
\newblock {\em Repr. Theory Appl. Categ.}, (10):vi+137, 2005.
\newblock Reprint of the 1982 original [Cambridge Univ. Press, Cambridge;
  MR0651714].

\bibitem[ML98]{MacLane:Working}
Saunders Mac~Lane.
\newblock {\em Categories for the working mathematician}, volume~5 of {\em
  Graduate Texts in Mathematics}.
\newblock Springer-Verlag, New York, second edition, 1998.

\bibitem[Per16]{Pervova:2016}
Ekaterina Pervova.
\newblock Diffeological vector pseudo-bundles.
\newblock {\em Topology Appl.}, 202:269--300, 2016.

\bibitem[Ros84]{Rosicky:1984}
J.~Rosick\'{y}.
\newblock Abstract tangent functors.
\newblock {\em Diagrammes}, 12:JR1--JR11, 1984.

\bibitem[Vin08]{Vincent:2008}
Martin Vincent.
\newblock Diffeological differential geometry.
\newblock Master's thesis, Department of Mathematical Sciences, University of
  Copenhagen, 2008.
\newblock Available at
  \url{https://www.math.ku.dk/english/research/tfa/top/paststudents/ms-theses/martinvincent.msthesis.pdf},
  downloaded on 4/15/2019.

\end{thebibliography}

\end{document}